\documentclass[12pt]{article}

\usepackage{amssymb,a4}
\usepackage{amsmath,amsfonts,amssymb,amsthm}

\newcommand{\Z}{{\mathbb Z}}
\newcommand{\C}{{\mathbb C}}

\newcommand{\N}{{\mathbb N}}

\newcommand{\E}{{\mathcal E}}

\def\<{\langle}
\def\>{\rangle}

\def\Aut{\mathrm{Aut}}
\def\End{\mathrm{End}}
\def\Hom{\mathrm{Hom}}

\newcommand{\la}{\langle}
\newcommand{\ra}{\rangle}

\newtheorem{thm}{Theorem}[section]
\newtheorem{prop}[thm]{Proposition}
\newtheorem{lem}[thm]{Lemma}

\newtheorem{rmk}[thm]{Remark}
\newtheorem{definition}[thm]{Definition}

\begin{document}

\begin{center}
{\Large \bf   $q$-Virasoro algebra and vertex algebras}
\end{center}

\begin{center}
{Hongyan Guo$^{a}$, Haisheng Li$^{b}$\footnote{Partially supported by NSA grant
H98230-11-1-0161 and China NSF grant (No. 11128103)}, Shaobin Tan$^{a}$\footnote{Partially supported by
China NSF grant (No.10931006) and a grant from the PhD Programs
Foundation of Ministry of Education of China (No.20100121110014).}
and Qing Wang$^{a}$\footnote{Partially supported by NSF of China (No.11371024),
Natural Science Foundation of Fujian Province (No.2013J01018) and
Fundamental Research Funds for the Central
University (No.2013121001).}\\
$\mbox{}^{a}$School of Mathematical Sciences, Xiamen University,
Xiamen 361005, China\\
$\mbox{}^{b}$ Department of Mathematical Sciences,\\
Rutgers University, Camden, NJ 08102, USA\\}
\end{center}

\begin{abstract}
In this paper, we study a certain deformation $D$ of the Virasoro algebra that was introduced and called $q$-Virasoro algebra by Nigro,
in the context of vertex algebras.
Among the main results, we prove that  for any complex number $\ell$, the category of restricted $D$-modules of level $\ell$
is canonically isomorphic to the category
of quasi modules for a certain vertex algebra of affine type.
We also prove that the category of restricted $D$-modules of level $\ell$ is canonically isomorphic to the category of
$\mathbb{Z}$-equivariant $\phi$-coordinated quasi modules for the same vertex algebra. In the process,
we introduce and employ a certain infinite dimensional Lie algebra which is defined in terms of generators and relations
and then identified explicitly with a subalgebra of  $\mathfrak{gl}_{\infty}$.
\end{abstract}

\section{Introduction}
\def\theequation{1.\arabic{equation}}
\setcounter{equation}{0}

Vertex algebras, as a new class of algebraic structures, have deep connections with classical algebras of various types,
including Lie algebras, associative algebras, and groups.
In particular, vertex algebras and their modules are often constructed and studied by using infinite-dimensional Lie algebras such as
affine Kac-Moody Lie algebras (including infinite-dimensional Heisenberg algebras) and the Virasoro algebras (cf. \cite{FZ}).
Infinite-dimensional Lie algebras of many other types, such as toroidal Lie algebras, quantum torus Lie algebras,
deformed Heisenberg Lie algebras, and Lie algebra $\mathfrak{gl}_{\infty}$,
 have also been canonically associated with vertex algebras or their likes (see \cite{BBS}, \cite{LTW}, \cite{Li1}, \cite{Li3}, \cite{JL}).

Note that in the association of vertex algebras and their modules to affine Lie algebras and the Virasoro algebra,
it is essential that their canonical generating functions
are mutually local in the sense that  for any two generating functions $a(x)$ and $b(x)$,
there is a nonnegative integer $k$ such that
$$(x_{1}-x_{2})^{k}a(x_{1})b(x_{2})=(x_{1}-x_{2})^{k}b(x_{2})a(x_{1}).$$
Behind this association is a conceptual result which was obtained in \cite{Li7},
stating that for any vector space $W$, every local subset of $\Hom (W,W((x)))$
generates a vertex algebra with $W$ as a module.

As for quantum torus Lie algebras, the situation is different;  their
 generating functions  are {\em not} local, instead they
 are  {\em quasi local} in the sense that for generating functions $a(x),b(x)$,
 there exists a nonzero polynomial $p(x_{1},x_{2})$ such that
 $$p(x_{1},x_{2})a(x_{1})b(x_{2})=p(x_{1},x_{2})b(x_{2})a(x_{1}).$$
 To associate vertex algebras to Lie algebras like quantum torus Lie algebras,
 a new conceptual construction of vertex algebras was established and
 a theory of quasi modules for vertex algebras was developed in \cite{Li1}.
 It was proved that for any vector space $W$, every quasi local subset of $\mbox{Hom}(W , W((x)))$ generates
 in a certain natural way a
vertex algebra with $W$ as a quasi module.
 For a vertex algebra $V$, the main feature for a quasi module is a weaker Jacobi identity axiom stating that for $u,v\in V$,
 there is a nonzero polynomial $p(x_{1},x_{2})$ such that the usual Jacobi identity multiplied by $p(x_{1},x_{2})$ holds.

 The theory of quasi modules was developed further in \cite{Li1} and \cite{Li2}, in which a notion of $\Gamma$-vertex algebra
 with $\Gamma$ a group  and a notion of quasi module for a $\Gamma$-vertex algebra
 were introduced.
  For a group $\Gamma$, a {\em $\Gamma$-vertex algebra} is simply a vertex algebra $V$
 equipped with a group representation $R: g\mapsto R_{g}$ of $\Gamma$ on $V$ and a linear character $\chi$ of $\Gamma$ such that
 \begin{eqnarray*}
 R_{g}({\bf 1})={\bf 1}, \  \  \  R_{g}Y(v,x)R_{g}^{-1}=Y(R_{g}v,\chi(g)^{-1}x),
\end{eqnarray*}
 for $g\in \Gamma,\  v\in V$.
 Furthermore, for a $\Gamma$-vertex algebra $V$, {\em a quasi $V$-module} is a quasi module $(W,Y_{W})$
 for $V$ viewed as a vertex algebra, satisfying the conditions that
 $$Y_{W}(R_{g}v,x)=Y_{W}(v,\chi(g)x)\   \  \  \mbox{ for }g\in \Gamma,\  v\in V$$
 and that for $u,v\in V$, there exist $g_{1},\dots,g_{r}\in \Gamma$ such that
 $$\left(\prod_{i=1}^{r}\left(x_{1}-\chi(g_{i})x_{2}\right)\right)[Y_{W}(u,x_{1}),Y_{W}(v,x_{2})]=0.$$

In a program to associate quantum vertex algebras to various quantum
algebras including quantum affine algebras, a theory of (weak) quantum vertex algebras and
their modules was developed in \cite{Li4}, \cite{Li5}, \cite{Li6}.
In particular in \cite{Li5}, for a weak quantum vertex algebra $V$ with an automorphism group $G$,
a notion of $G$-equivariant $\phi$-coordinated quasi module was introduced.
Quantum vertex algebras in this sense are generalizations of vertex algebras and vertex superalgebras,
so that all the results directly apply to vertex algebras and vertex superalgebras.

In this paper, we study in the context of vertex algebras a certain $q$-deformation of the Virasoro algebra,
which was introduced by Nigro (see \cite{N}).  The $q$-Virasoro algebra, introduced by Nigro and denoted by $D$ in this paper,
 by definition is the Lie algebra with generators ${\bf c}$ and $D^{\alpha}(n)$  with $\alpha,n\in\mathbb{Z}$,
subject to relations $D^{-\alpha}(n)=-D^{\alpha}(n)$ and
\begin{eqnarray}
 [D^{\alpha}(n),D^{\beta}(m)] & = &
    (q-q^{-1})[\alpha m-\beta n]_{q}D^{\alpha+\beta}(m+n)
                                                    \nonumber\\
    &&{} -(q-q^{-1})[\alpha m+\beta n]_{q}D^{\alpha-\beta}(m+n)
                                                               \nonumber\\
    &&{} +([m]_{q^{{\alpha}+{\beta}}}-[m]_{q^{{\alpha}-{\beta}}})\delta_{m+n,0}{\bf c}\label{eq:1.1}
\end{eqnarray}
for $\alpha,\beta,m,n\in \mathbb{Z},$
where ${\bf c}$ is a central element, $q$ is a complex parameter which is neither zero nor a root of unity, and
$[n]_{q} = \frac{q^{-n}-q^{n}}{q^{-1}-q} $. Note that $[n]_{1}=n$ by convention.
As our main results, we establish a canonical connection between  this $q$-Virasoro algebra
and $\Gamma$-vertex algebras  and their quasi modules in the sense of \cite{Li1}, \cite{Li2} and \cite{Li3}, and we also
associate $q$-Virasoro algebra with vertex algebras in terms of $\mathbb{Z}$-equivariant
$\phi$-coordinated quasi modules in the sense of \cite{Li5}.

 We now describe the content of this paper with some technical details.
First, just as with affine Lie algebras, for each $\alpha\in \Z$ we form a generating function
$$D^{\alpha}(x) = \sum_{n\in\Z}D^{\alpha}(n)x^{-n-1}.$$
A $D$-module $W$ is said to be {\em restricted}  if for any $\alpha\in \Z,\ w\in W$, $D^{\alpha}(n)w=0$
for $n$ sufficiently large, and to be {\em of level} $\ell\in \C$ if the central
element ${\bf c}$ acts on $W$ as scalar $\ell$.
Writing the commutation relation (\ref{eq:1.1}) in terms of generating functions, one sees that
$D^{\alpha}(x)$ with $\alpha\in \Z$ form
a quasi local subset $U_{W}$ of $\E(W)$ for any restricted $D$-module $W$ of level $\ell$.
In view of  \cite{Li1}, $U_{W}$ generates a vertex algebra with $W$ as a quasi module.
This gives a {\em conceptual} association of vertex algebras to Lie algebra $D$.
The core of this paper is {\em explicitly} to construct the desired vertex algebras  and then associate them to the Lie algebra $D$.

Second, as a key step we modify the generating functions $D^{\alpha}(x)$ for $\alpha\in \Z$ to introduce
\begin{eqnarray}
\widetilde{D}^{\alpha}(x)=\begin{cases}D^{\alpha}(x)&\ \ \mbox{ if }\alpha=0\\
D^{\alpha}(x)-\frac{1}{q^{-\alpha}-q^{\alpha}}{\bf c}x^{-1}&\ \ \mbox{ if }\alpha\ne 0.
\end{cases}
\end{eqnarray}
The main reason for doing this is that commutation relation for $\widetilde{D}^{\alpha}(x)$ takes a better shape.
Let $W$ be an arbitrary restricted $D$-module of level $\ell$.
Set $\widetilde{U}_{W}=\{ \widetilde{D}^{\alpha}(q^{r}x)\  | \ \alpha,r\in \Z\}$.
This set is still quasi local, so it generates a vertex algebra $\la\widetilde{U}_{W}\ra$
with $W$ as faithful quasi module. A detailed study on this vertex algebra naturally leads us to
a  ``new''  infinite-dimensional Lie algebra.

Third, we introduce a Lie algebra denoted by ${\mathfrak{D}}$, which by
 definition is the Lie algebra  generated by
$d^{\alpha,r}$ for $\alpha,r\in\mathbb{Z}$,
subject to  relations: $d^{-\alpha,r}=-d^{\alpha,r}$ and
\begin{eqnarray*}
 \quad\qquad\qquad [d^{\alpha,r},d^{\beta,s}] &=
   & \delta_{\alpha+\beta,s-r}d^{\alpha+\beta,-\alpha+s}
   -\delta_{\alpha+\beta,r-s}d^{\alpha+\beta,\alpha+s}
                         \nonumber\\
   &&{} -\delta_{\alpha-\beta,s-r}d^{\alpha-\beta,-\alpha+s}
   +\delta_{\alpha-\beta,r-s}d^{\alpha-\beta,\alpha+s}.
\end{eqnarray*}
 On ${\mathfrak{D}}$,
there is a non-degenerate symmetric invariant bilinear form $\la\cdot,\cdot\ra$ defined by
\begin{eqnarray*}
\langle d^{\alpha,r},d^{\beta,s}\rangle =\delta_{r,s}\left(\delta_{\alpha-\beta,0}-\delta_{\alpha+\beta,0}\right)
\   \   \   \mbox{ for } \alpha,\beta,r, s\in\mathbb{Z}.
\end{eqnarray*}
Then we have an affine Lie algebra
 $\widehat{\mathfrak{D}}={\mathfrak{D}}\otimes\mathbb{C}[t,t^{-1}]\oplus \C {\bf c}$.
 Furthermore, for any complex number $\ell$, we have a vertex algebra $V_{\widehat{\mathfrak{D}}}(\ell,0)$,
 whose underlying vector space is the level $\ell$ generalized Verma module (or Weyl module) of $\widehat{\mathfrak{D}}$.

 Vertex algebra $V_{\widehat{\mathfrak{D}}}(\ell,0)$ actually is a $\Gamma$-vertex algebra.
 For $m\in \Z$, there is an automorphism $\sigma_{m}$ of $\widehat{\mathfrak{D}}$, uniquely determined by
  $$\sigma_{m}(d^{\alpha,r})=d^{\alpha,r+m}\  \  \  \mbox{ for }\alpha,r,m\in \Z.$$
 This automorphism of $\widehat{\mathfrak{D}}$ gives rise to
 an automorphism of $V_{\widehat{\mathfrak{D}}}(\ell,0)$, which is also denoted by $\sigma_{m}$.
 Then we have an automorphism group $\Gamma=\{ \sigma_{m} \ |\  m\in \Z\}$ of $V_{\widehat{\mathfrak{D}}}(\ell,0)$,
 which is naturally isomorphic to $\Z$.
From \cite{Li1}, $V_{\widehat{\mathfrak{D}}}(\ell,0)$ has a natural $\Gamma$-vertex algebra structure with the
linear character
 $\chi: \Gamma\rightarrow \C^{\times};\   \   \sigma_{m}\mapsto q^{m}\  \  (m\in \Z)$.
We prove that  on a vector space $W$,
a restricted $D$-module structure of level $\ell$
 is equivalent to a  quasi $V_{\widehat{\mathfrak{D}}}(\ell,0)$-module structure.
In this way, we obtain an equivalence of categories.

Another interesting way to modify the generating functions $D^{\alpha}(x)$ for $\alpha \in \Z$ is to set
$$\hat{D}^{\alpha}(x)=x\widetilde{D}^{\alpha}(x) = \sum\limits_{n\in\mathbb{Z}}\widetilde{D}^{\alpha}(n)x^{-n}.$$
By using \cite{Li4}, we prove that  for any  restricted $D$-module $W$ of level $\ell$,
there is a structure of a $\mathbb{Z}$-equivariant $\phi$-coordinated quasi module
for the vertex algebra  $V_{\widehat{\mathfrak{D}}}(\ell,0)$, which is uniquely determined by
$$Y_{W}(d^{\alpha,r},x)=\hat{D}^{\alpha}(q^{r}x)\  \  \   \mbox{ for }\alpha,r\in \Z.$$
(Note that $\mathbb{Z}$ is
considered as an automorphism group of $V_{\widehat{\mathfrak{D}}}(\ell,0)$ as  above.)
On the other hand, we show that any $\mathbb{Z}$-equivariant $\phi$-coordinated quasi module
for the vertex algebra  $V_{\widehat{\mathfrak{D}}}(\ell,0)$ is naturally a restricted $D$-module of level $\ell$.

This paper is organized as follows: In Section 2,
 we study $q$-Virasoro algebra $D$ in the context of $\Gamma$-vertex algebras
and their quasi modules.
In Section 3, we study the $q$-Virasoro algebra $D$ in terms of vertex algebras and
$\Gamma$-equivariant $\phi$-coordinated quasi modules.

\section{Associating $q$-Virasoro algebra with $\Gamma$-vertex algebras in terms of quasi modules}
\label{Sect:V(k,0)}
\def\theequation{2.\arabic{equation}}
\setcounter{equation}{0}

In this section,  we recall from \cite{Li1} and \cite{Li3} some basic results on quasi modules for
vertex algebras and from \cite{N} the $q$-Virasoro algebra $D$. We define a Lie algebra ${\mathfrak{D}}$ with a symmetric invariant bilinear form
and show that the vertex algebra $V_{\widehat{\mathfrak{D}}}(\ell,0)$ associated to the affine Lie algebra
$\widehat{\mathfrak{D}}$ of level $\ell$ is a $\Gamma$-vertex algebra with $\Gamma=\Z$.
Then we give an isomorphism between the category of restricted $D$-modules of level $\ell$
and that of quasi $V_{\widehat{\mathfrak{D}}}(\ell,0)$-modules.

Throughout this paper, $\N$ denotes the set of nonnegative integers,
 $\mathbb{C}^{\times}$ denotes the multiplicative group of nonzero complex numbers (while $\C$ denotes the complex number field), and
the symbols $x,y,x_{1},x_{2},\dots $ denote mutually commuting independent formal variables. All vector spaces in this paper are
considered to be over  $\mathbb{C}$. For a vector space $U$, $U((x))$ is the vector space of lower
truncated integral power series in $x$ with coefficients
 in $U$, $U[[x]]$ is the vector space of nonnegative integral
 power series in $x$ with coefficients in $U$, and
$U[[x,x^{-1}]]$ is the vector space of doubly infinite integral
 power series in $x$ with coefficients in $U$ .



We first recall the definition of a vertex algebra (cf. \cite{LL}).

\begin{definition}
{\em A {\em vertex algebra} is a vector space $V$ equipped with a linear map
             $$Y(\cdot,x) :V\longrightarrow \mathrm{Hom}(V,V((x)))\subset \mathrm{EndV}[[x,x^{-1}]]$$
              $$ v\longmapsto Y(v,x)=\sum_{n\in\mathbb{Z}}v_{n}x^{-n-1}\ \ (\mbox{where }v_{n}\in \End V)$$
              and with a distinguished vector $\textbf{1}\in V$, called the {\em vacuum vector},
              such that all the following conditions are satisfied for $u,v\in V$:
              $$Y(\textbf{1},x)v=v\;\;\mbox{(the {\em vacuum property})},$$
              $$Y(v,x)\textbf{1}\in V[[x]] \;\;\mbox{and}\;\; \lim_{x\mapsto 0}Y(v,x)\textbf{1} = v \;\;\mbox{(the {\em creation property})},$$
              and
              \begin{eqnarray*}
              &&x_{0}^{-1}\delta\left(\frac{x_{1}-x_{2}}{x_{0}}\right)Y(u,x_{1})Y(v,x_{2}) -
                         x_{0}^{-1}\delta\left(\frac{x_{2}-x_{1}}{-x_{0}}\right)Y(v,x_{2})Y(u,x_{1})\\
               &&\hspace{2cm}= x_{2}^{-1}\delta\left(\frac{x_{1}-x_{0}}{x_{2}}\right)Y(Y(u,x_{0})v,x_{2})
              \end{eqnarray*}
              (the {\em Jacobi identity}).}
\end{definition}

The following notion of quasi module for vertex algebras was introduced in \cite{Li1}:

\begin{definition}
{\em Let $V$ be a vertex algebra. A {\em quasi $V$-module} is a vector space
                  $W$ equipped with a linear map $Y_{W}(\cdot,x)$ from $V$ to $\mathrm{Hom}(W,W((x)))$ such that
                   $$ Y_{W}(1,x) = 1_{W}\ \ (\mbox{the identity operator on }W)$$
                    and  such that for any $u,v\in V$, there exists a nonzero polynomial $f(x_{1},x_{2})$
                    such that
                    \begin{eqnarray}
                    &&{}x_{0}^{-1}\delta\left(\frac{x_{1}-x_{2}}{x_{0}}\right)f(x_{1},x_{2})Y_{W}(u,x_{1})Y_{W}(v,x_{2})
                    \nonumber\\
                    &&\hspace{1cm} -x_{0}^{-1}\delta\left(\frac{x_{2}-x_{1}}{-x_{0}}\right)f(x_{1},x_{2})Y_{W}(v,x_{2})Y_{W}(u,x_{1})
                                                               \nonumber\\
                   &= & x_{2}^{-1}\delta\left(\frac{x_{1}-x_{0}}{x_{2}}\right)f(x_{1},x_{2})Y_{W}(Y(u,x_{0})v,x_{2}). \nonumber
\end{eqnarray}}
\end{definition}

We next recall from \cite{Li1} and \cite{Li2} a conceptual construction of vertex algebras and
their quasi modules. Let $W$ be a general vector space. Set
\begin{eqnarray}
\mathcal{E}(W)=\mbox{Hom}(W , W((x)))\subset(\mbox{EndW})[[x,x^{-1}]].     \label{eq:2.1}
\end{eqnarray}
The identity operator on $W$, denoted by $\textbf{1}_{W}$, is a special
element of $\mathcal{E}(W).$

The following generalization of the notion of locality was introduced in \cite{Li1}:

\begin{definition} {\em Formal series $a(x),b(x)\in\mathcal{E}(W)$ are said to be
                 {\em mutually quasi local} if there exists a nonzero polynomial $f(x_{1},x_{2})$ such that
                 \begin{eqnarray}
                 f(x_{1},x_{2})a(x_{1})b(x_{2})=f(x_{1},x_{2})b(x_{2})a(x_{1}).  \label{eq:2.2}
                 \end{eqnarray}
                 A subset (subspace) $U$ of $\mathcal{E}(W)$ is said to be {\em quasi local} if any
                 $a(x),b(x)\in U$ are mutually quasi local.}
\end{definition}

The following notion was due to \cite{G-K-K}:

  \begin{definition}
  {\em Let $\Gamma$ be a subgroup of $\mathbb{C}^{\times}$. Formal series $a(x),b(x)\in\mathcal{E}(W)$ are said to be
                 {\em mutually $\Gamma$-local} if there exists a (nonzero) polynomial
                 $$f(x_{1},x_{2})\in\langle(x_{1}-\alpha x_{2})\ | \ \alpha\in\Gamma\rangle\subset\mathbb{C}[x_{1},x_{2}]$$
                 such that (\ref{eq:2.2}) holds.  Furthermore, the notion of $\Gamma$-local subset (space) is defined in the obvious way.}
\end{definition}

Denote by $\C(x_{1},x_{2})$ the field of rational functions. Let
$$\iota_{x_{1},x_{2}}: \  \C(x_{1},x_{2})\rightarrow \C((x_{1}))((x_{2}))$$
be the canonical extension of the ring embedding of $\C[x_{1},x_{2}]$ into the field $\C((x_{1}))((x_{2}))$.
In particular, for $\alpha\in \C,\ m\in \Z$ we have
$$\iota_{x_{1},x_{2}}\left((x_{1}-\alpha x_{2})^{m}\right)=\sum_{j\ge 0}\binom{m}{j}(-\alpha)^{j}x_{1}^{m-j}x_{2}^{j}.$$

Let $W$ be a vector space as before. Let $U$ be any quasi local subset of $\mathcal{E}(W)$ and let $a(x),b(x)\in U$.
Notice that the relation (\ref{eq:2.2}) implies
 \begin{eqnarray}
 f(x_{1},x_{2})a(x_{1})b(x_{2})\in \mbox{Hom}(W, W((x_{1},x_{2}))).     \label{eq:2.3}
 \end{eqnarray}
Let $\alpha\in \mathbb{C}^{\times}$. Define $a(x)_{(\alpha,n)}b(x)\in\mbox{(End W)}[[x,x^{-1}]]$ for $n\in \Z$
                  in terms of generating function
                 \begin{eqnarray}\mathcal{Y}_{\alpha}(a(x),x_{0})b(x)=
                  \sum_{n\in\mathbb{Z}}(a(x)_{(\alpha,n)}b(x))x_{0}^{-n-1}   \label{eq:2.4}\end{eqnarray}
                  by
                 \begin{eqnarray}\mathcal{Y}_{\alpha}(a(x),x_{0})b(x) =
                  l_{x,x_{0}}(f(x_{0}+\alpha x,x)^{-1})(f(x_{1},x)a(x_{1})b(x))|_{x_{1}=\alpha x+x_{0}},  \label{eq:2.5}\end{eqnarray}
                 where $f(x_{1},x_{2})$ is any nonzero polynomial such that (\ref{eq:2.3}) holds.

Let $\Gamma$ be a subgroup of $\mathbb{C}^{\times}$.
A quasi local subspace $U$ of $\mathcal{E}(W)$ is said to be {\em $\mathcal{Y}_{\Gamma}$-closed} if
                 \begin{eqnarray} a(x)_{(\alpha,n)}b(x)\in U             \label{eq:2.6}\end{eqnarray}
                  for all $a(x),b(x)\in U,\  \alpha\in\Gamma,\  n\in\mathbb{Z}$.
                  In the case $\Gamma=\{1\}$, we say $U$ is {\em $\mathcal{Y}_{1}$-closed} and we particularly set
                   \begin{eqnarray}
                   \mathcal{Y}(a(x),x_{0})b(x)&=&\mathcal{Y}_{1}(a(x),x_{0})b(x),\\
                 a(x)_{(n)}b(x)&=&a(x)_{(1,n)}b(x).
                 \end{eqnarray}

Let $S$ be any quasi local subset of $\mathcal{E}(W)$ and let $\Gamma$ be any subgroup  of $\mathbb{C}^{\times}$.
Denote by $\langle S\rangle_{\Gamma}$ the smallest $\mathcal{Y}_{\Gamma}$-$closed$
quasi local subspace of $\mathcal{E}(W)$, which contains $S$ and $\textbf{1}_{W}$.
From \cite{Li1} (Proposition 4.10),
$\langle S\rangle_{\Gamma}$ is linearly spanned by the vectors
$$a^{(1)}(x)_{(\alpha_{1},n_{1})}\cdots a^{(r)}(x)_{(\alpha_{r},n_{r})}\textbf{1}_{W}$$
for $r\in \N, \  a^{(i)}(x)\in S,\  \alpha_{i}\in\Gamma,\   n_{i}\in\mathbb{Z}$.
Moreover, we set $\langle S\rangle=\langle S\rangle_{\{1\}},$ the smallest $\mathcal{Y}_{1}$-closed quasi local subspace
of $\mathcal{E}(W)$,  which contains $S$ and $\textbf{1}_{W}$.

For any $\alpha\in \C^{\times}$ and $a(x)\in\mathcal{E}(W)$,
it is clear that $a(\alpha x)\in\mathcal{E}(W).$
Following \cite{Li1}, we define
$\overline{R}_{\alpha}\in \mbox{End}(\mathcal{E}(W))$ by
\begin{eqnarray}
 \overline{R}_{\alpha}(a(x) )= a(\alpha x)\ \ \  \mbox{ for }  a(x)\in\mathcal{E}(W).
 \end{eqnarray}
Then the map $\overline{R}$ : $\mathbb{C}^{\times}\rightarrow \mbox{End}(\mathcal{E}(W)),$
sending $\alpha$ to $ \overline{R}_{\alpha}$ for $\alpha\in \C^{\times}$, is a representation of $\mathbb{C}^{\times}$
on $\mathcal{E}(W).$

The following result was obtained in \cite{Li1} (Proposition 4.12 and Theorem 5.3):

\begin{thm}\label{thm.2.4}
Let $W$ be a vector space and let $\Gamma$ be a subgroup of $\mathbb{C}^{\times}$.
           Then for any quasi local subset $S$ of $\mathcal{E}(W)$,
           $(\langle S\rangle_{\Gamma},\mathcal{Y}_{1},1_{W})$ carries the structure
           of a vertex algebra and $W$ is a quasi module with $Y_{W}(a(x),z)=a(z)$ for $a(x)\in \langle S\rangle_{\Gamma}$.
           Furthermore, if the subspace spanned by $S$ is a $\Gamma$-local subspace and a $\Gamma$-submodule
           of $\mathcal{E}(W)$, then $\langle S\rangle=\langle S\rangle_{\Gamma}.$
\end{thm}

We now recall the definitions of a $\Gamma$-vertex algebra and a quasi module (see \cite{Li2}).

\begin{definition}\label{def.2.6}
{\em Let $\Gamma$ be an abstract group (which is not necessarily a subgroup of $\C^{\times}$).
A {\em $\Gamma$-vertex algebra} is a vertex algebra $V$ equipped with group homomorphisms
                $$ R: \Gamma \rightarrow {\mathrm{GL}}(V); \ g \longmapsto R_{g} \;\;\mbox{and}\;\;
                 \varphi: \Gamma \rightarrow \mathbb{C}^{\times}$$
                 such that $R_{g}(\textbf{1}) = \textbf{1}$ for $g\in\Gamma$
                 and
                 $$R_{g}Y(v,x)R_{g}^{-1} = Y(R_{g}(v), \varphi(g)^{-1}x) \;\;\mbox{for}\; g\in\Gamma, \  v\in V.$$}
\end{definition}

\begin{rmk} {\em This notion is equivalent to that of $\Gamma$-vertex algebra defined in \cite{Li1}. }
\end{rmk}

\begin{definition}
{\em Let $V$ be a $\Gamma$-vertex algebra. A {\em quasi $V$-module} is a quasi module $(W, Y_{W})$
                for $V$ viewed as a vertex algebra, satisfying the condition that
                $$Y_{W}(R_{g}(v),x) = Y_{W}(v, \varphi(g)x) \;\;\mbox{for}\; g\in\Gamma,\  v\in V$$
                and for $u,v\in V$, there exist $\alpha_{1},\ldots,
                \alpha_{k}\in\varphi(\Gamma)\subset\mathbb{C}^{\times}$ such that
                $$(x_{1}-\alpha_{1}x_{2})\cdots(x_{1}-\alpha_{k}x_{2})[Y_{W}(u,x_{1}), Y_{W}(v,x_{2})] = 0.$$}
\end{definition}

\begin{rmk}\label{rgradedva-gammava}
{\em Let $V$ be a {\em $\Z$-graded vertex algebra} in the sense that $V$ is a vertex algebra
equipped with a $\Z$-grading $V=\oplus_{n\in \Z}V_{(n)}$ such that ${\bf 1}\in V_{(0)}$ and
$$u_{m}V_{(n)}\subset V_{(k+n-m-1)}\ \ \  \mbox{ for }u\in V_{(k)},\ k,m,n\in \Z.$$
Denote by $L(0)$ the linear operator on $V$, defined by $L(0)|_{V_{(n)}}=n$ for $n\in \Z$.
Define an {\em automorphism of a $\Z$-graded vertex algebra $V$} to be an automorphism
of vertex algebra $V$, which  preserves the $\Z$-grading.
Let $\Gamma$ be an automorphism group of a $\Z$-graded vertex algebra $V$ and
let $\varphi: \Gamma \rightarrow \mathbb{C}^{\times}$ be a group homomorphism.
Then it is straightforward to show (cf. \cite{Li2}) that $V$ becomes a $\Gamma$-vertex algebra
with $R_{g}=\varphi(g)^{-L(0)}g$ for $g\in \Gamma$.}
\end{rmk}

As we shall need later, we next recall the $\Z$-graded vertex algebras associated to affine Lie algebras.
 Let ${\mathfrak{g}}$ be any (possibly infinite-dimensional) Lie algebra equipped with a symmetric invariant
bilinear form $\langle\cdot,\cdot\rangle$. To the pair $({\mathfrak{g}},\langle\cdot,\cdot\rangle)$
one has an affine Lie algebra
$$\hat{\mathfrak{g}}={\mathfrak{g}}\otimes \C[t,t^{-1}]\oplus \C {\bf c},$$
where ${\bf c}$ is central and
$$[a\otimes t^{m},b\otimes t^{n}]=[a,b]\otimes t^{m+n}+m\delta_{m+n,0}\langle a,b\rangle {\bf c}$$
for $a,b\in {\mathfrak{g}},\  m,n\in \Z$.
 Let $\ell$ be any complex number and  denote by $\C_{\ell}$ the one-dimensional $({\mathfrak{g}}\otimes \C[t]+\C {\bf c})$-module $\C$ with
 ${\mathfrak{g}}\otimes \C[t]$ acting trivially and with ${\bf c}$ acting as scalar $\ell$.
 Form an induced $\hat{\mathfrak{g}}$-module of level $\ell$
 $$V_{\hat{\mathfrak{g}}}(\ell,0)=U(\hat{\mathfrak{g}})\otimes_{U({\mathfrak{g}}\otimes \C[t]+\C {\bf c})} \C_{\ell}.$$
 Set ${\bf 1}=1\otimes 1$ and identify $a\in {\mathfrak{g}}$ with $a(-1){\bf 1}\in V_{\hat{\mathfrak{g}}}(\ell,0)$, making
 ${\mathfrak{g}}$ a subspace of $V_{\hat{\mathfrak{g}}}(\ell,0)$, where for $n\in \Z$, $a(n)$ denotes $a\otimes t^{n}$ alternatively.
 Then there exists a vertex algebra structure on $V_{\hat{\mathfrak{g}}}(\ell,0)$, uniquely determined by
 the conditions that ${\bf 1}$ is the vacuum vector and that
 $$Y(a,x)=a(x)=\sum_{n\in \Z}a(n)x^{-n-1}$$ for $a\in {\mathfrak{g}}$.
 Furthermore, $V_{\hat{\mathfrak{g}}}(\ell,0)$ is a $\Z$-graded vertex algebra  with
$$V_{\hat{\mathfrak{g}}}(\ell,0)_{(n)}=0\  \mbox{ for }n<0,\  V_{\hat{\mathfrak{g}}}(\ell,0)_{(0)}=\C {\bf 1},\  \
V_{\hat{\mathfrak{g}}}(\ell,0)_{(1)}={\mathfrak{g}},$$
where ${\mathfrak{g}}$ generates $V_{\hat{\mathfrak{g}}}(\ell,0)$ as a vertex algebra.

\begin{lem}\label{lautomorphism}
Let $\sigma$ be an automorphism of Lie algebra ${\mathfrak{g}}$, which preserves the bilinear form $\langle\cdot,\cdot\rangle$.
Then $\sigma$ extends uniquely to an automorphism
of the $\Z$-graded vertex algebra $V_{\hat{\mathfrak{g}}}(\ell,0)$.
\end{lem}

\begin{proof}
First of all, lift $\sigma$ to an automorphism $\hat{\sigma}$ of the affine Lie algebra $\hat{\mathfrak{g}}$ by
$$\hat{\sigma}(a\otimes t^{n}+\mu {\bf c})=\sigma(a)\otimes t^{n}+\mu {\bf c}
\  \  \  \mbox{for }a\in {\mathfrak{g}},\  n\in \Z,\  \mu\in \C.$$
This induces an automorphism of the universal enveloping algebra $U(\hat{\mathfrak{g}})$,
 also denoted by $\hat{\sigma}$.
 As $\hat{\sigma}$ preserves the subalgebra $({\mathfrak{g}}\otimes \C[t]+\C {\bf c})$,
 $\hat{\sigma}$ gives rise to a linear automorphism $\tilde{\sigma}$ of $ V_{\hat{\mathfrak{g}}}(\ell,0)$
             such that $\tilde{\sigma}(\textbf{1})=\textbf{1}$ and
             $\tilde{\sigma}(Xv)=\hat{\sigma}(X)\tilde{\sigma}(v)$
             for $X\in U(\widehat{\mathfrak{g}}),\ v\in V_{\hat{\mathfrak{g}}}(\ell,0).$ In particular, we have
             $$\tilde{\sigma}(a_{n}v)=\tilde{\sigma}(a(n)v)=\hat{\sigma}(a(n))\tilde{\sigma}(v)
             =\sigma(a)(n)\tilde{\sigma}(v)=\sigma(a)_{n}\tilde{\sigma}(v)$$
             for $a\in {\mathfrak{g}},\ n\in \Z$.
As $ V_{\hat{\mathfrak{g}}}(\ell,0)$ is generated by ${\mathfrak{g}}$ as a vertex algebra,
             it follows that $\tilde{\sigma}$ is an automorphism of vertex algebra $V_{\hat{\mathfrak{g}}}(\ell,0)$.
 It is clear that $\tilde{\sigma}$ preserves the $\Z$-grading. Thus
 $\tilde{\sigma}$ is an automorphism of the $\Z$-graded vertex algebra $V_{\hat{\mathfrak{g}}}(\ell,0)$.
Furthermore,  we have
$$\tilde{\sigma}(a)=\tilde{\sigma}(a_{-1}{\bf 1})=\sigma(a)_{-1}{\bf 1}=\sigma(a)\  \  \mbox{ for }a\in {\mathfrak{g}}.$$
That is, $\tilde{\sigma}$ extends $\sigma$.  On the other hand,
as $V_{\hat{\mathfrak{g}}}(\ell,0)$ is generated by ${\mathfrak{g}}$ as a vertex algebra, the extension of $\sigma$  is unique.
\end{proof}

In view of Lemma \ref{lautomorphism},
for any automorphism group $G$ of Lie algebra ${\mathfrak{g}}$,
which preserves the bilinear form $\langle\cdot,\cdot\rangle$, we can and we shall consider $G$ as an automorphism group
of the $\Z$-graded vertex algebra $V_{\hat{\mathfrak{g}}}(\ell,0)$.
We shall simply use $\sigma\in G$ for its extension.

We recall the following $q$-analog of the Virasoro algebra from \cite{N}.

\begin{definition}
{\em Let $q$ be a nonzero complex number. Denote by
$D$ the Lie algebra with generators ${\bf c}$ and $D^{\alpha}(n)\ (\alpha,n\in\mathbb{Z})$,
subjects to relations
$$D^{-\alpha}(n)=-D^{\alpha}(n),$$
 \begin{eqnarray}\label{eDbracket}
 [D^{\alpha}(n),D^{\beta}(m)] & = &
    (q-q^{-1})[\alpha m-\beta n]_{q}D^{\alpha+\beta}(m+n)
                                                    \nonumber\\
    &&{} -(q-q^{-1})[\alpha m+\beta n]_{q}D^{\alpha-\beta}(m+n)
                                                               \nonumber\\
    &&{} +([m]_{q^{{\alpha}+{\beta}}}-[m]_{q^{{\alpha}-{\beta}}})\delta_{m+n,0}{\bf c}
\end{eqnarray}
for $m,n,\alpha,\beta\in\Z$,
where ${\bf c}$ is a central element and
$[n]_{q}$ is the $q$-integer defined by
        $$[n]_{q}= \frac{q^{-n}-q^{n}}{q^{-1}-q}.$$}
\end{definition}

\begin{rmk}\label{rlimit-cases} {\em Note that  as a convention in quantum algebra theory, it is understood that $[n]_{q}=n$
for $q=\pm 1$. This is consistent with the fact that
$$\lim_{q\rightarrow \pm 1}[n]_{q}=n.$$
The following are two simple properties for $q$-integers:
\begin{eqnarray}\label{etwo-facts}
[-n]_{q}=-[n]_{q}, \  \   \  \ [n]_{q^{-1}}=[n]_{q}.
\end{eqnarray}}
\end{rmk}

\begin{rmk}\label{rDbasis} {\em From the first relation in the definition of $D$, we have $D^{0}(n)=0$ for all $n\in \Z$ and
we see that $D$ is linearly spanned by vectors ${\bf c}$ and $D^{\alpha}(n)$ for $\alpha,n\in \Z$ with $\alpha\ge 1$.
In principle, one can show that these vectors form a basis as follows:
First, define a non-associative algebra $\tilde{D}$ with a basis $\{ D(\alpha,n)\  | \  \alpha,n\in \Z\}\cup \{ {\bf c}\}$
and with an operation $[\cdot,\cdot]$ defined by (\ref{eDbracket}) and by $[{\bf c},\tilde{D}]=0=[\tilde{D},{\bf c}]$.
Let $J$ be the linear span of $D(-\alpha,n)+D(\alpha,n)$ for $\alpha,n\in \Z$.
From (\ref{eDbracket}), using (\ref{etwo-facts})  we have
\begin{eqnarray*}
&&[D(\alpha,n),D(\beta,m)+D(-\beta,m)]=0,\\
&&[D(\alpha,n),D(\beta,m)]+[D(\beta,m),D(\alpha,n)]\in J
\end{eqnarray*}
for $\alpha,\beta,m,n\in \Z$. It follows that $J$  is a two-sided ideal.
It is clear that the operation $[\cdot,\cdot]$ on the quotient algebra $\tilde{D}/J$
is skew symmetric. Next, the Jacobi identity is straightforward to check, thus $\tilde{D}/J$ is a Lie algebra.
Then it follows that $D\simeq \tilde{D}/J$, which implies that $D$ has the desired basis.}
\end{rmk}

\begin{rmk}\label{rlimit2}
{\em Note that  for $f(x)\in \C[[x,x^{-1}]]$, we have
 \begin{eqnarray}
 \lim_{\mu \rightarrow 0}\frac{1}{q^{-\mu}-q^{\mu}}(f(q^{-\mu}x)-f(q^{\mu}x))=x\frac{d}{dx}f(x).
 \end{eqnarray}
In particular, we have
       \begin{eqnarray}
 \lim_{\mu \rightarrow 0}\frac{1}{q^{-\mu}-q^{\mu}}
\left[\delta\left(\frac{q^{\mu}x_{2}}{x_{1}}\right)
-\delta\left(\frac{q^{-\mu}x_{2}}{x_{1}}\right)\right]=-x_{2}\frac{\partial}{\partial x_{2}}\delta\left(\frac{x_{2}}{x_{1}}\right).
 \end{eqnarray}   }
\end{rmk}

[Note that the limit of the Lie algebra $D$ for $q\rightarrow \pm 1$ is an abelian Lie algebra.
Remark: Set $\bar{D}^{\alpha}(n)=\frac{1}{q-q^{-1}}D^{\alpha}(n)$ for $\alpha,n\in \Z$.  Then
\begin{eqnarray*}\label{eDbracket}
 [\bar{D}^{\alpha}(n),\bar{D}^{\beta}(m)] & = &
    [\alpha m-\beta n]_{q}\bar{D}^{\alpha+\beta}(m+n) -[\alpha m+\beta n]_{q}\bar{D}^{\alpha-\beta}(m+n)
                                                               \nonumber\\
    &&{} +(q-q^{-1})^{-2}([m]_{q^{{\alpha}+{\beta}}}-[m]_{q^{{\alpha}-{\beta}}})\delta_{m+n,0}{\bf c}.
\end{eqnarray*}
Taking the limit $q\rightarrow 1$, we get
\begin{eqnarray*}\label{eDbracket}
 [\bar{D}^{\alpha}(n),\bar{D}^{\beta}(m)] & = &
    (\alpha m-\beta n)\bar{D}^{\alpha+\beta}(m+n)
  -(\alpha m+\beta n)\bar{D}^{\alpha-\beta}(m+n)
                                                               \nonumber\\
    &&{} +\lim_{q\rightarrow 1}(q-q^{-1})^{-2}([m]_{q^{{\alpha}+{\beta}}}-[m]_{q^{{\alpha}-{\beta}}})\delta_{m+n,0}{\bf c}.
\end{eqnarray*}
(Calculate the limit in the central extension part.)
The limit algebra is a known Lie algebra.]

From now on, we shall assume that $q$ is {\em not a root of unity.}

For $\alpha\in\mathbb{Z},$ form a generating function
\begin{eqnarray}
D^{\alpha}(x) = \sum_{n\in\mathbb{Z}}D^{\alpha}(n)x^{-n-1}.
\end{eqnarray}
For $\alpha,\beta\in \Z$, we have
\begin{eqnarray}\label{eDbracket-generating}
 &&[D^{\alpha}(x_{1}),D^{\beta}(x_{2})] \nonumber\\
 &=&\sum_{m,n\in\mathbb{Z}}[D^{\alpha}(n),D^{\beta}(m)]x_{1}^{-n-1}x_{2}^{-m-1}
                                                            \nonumber\\
&=&q^{-\alpha}D^{\alpha+\beta}(q^{-\alpha}x_{2})x_{1}^{-1}\delta\left(\frac{q^{-\alpha}x_{2}}{q^{\beta}x_{1}}\right)
 -q^{\alpha}D^{\alpha+\beta}(q^{\alpha}x_{2})x_{1}^{-1}\delta\left(\frac{q^{\alpha}x_{2}}{q^{-\beta}x_{1}}\right)
                                                            \nonumber\\
&&{} -q^{-\alpha}D^{\alpha-\beta}(q^{-\alpha}x_{2})x_{1}^{-1}\delta\left(\frac{q^{-\alpha}x_{2}}{q^{-\beta}x_{1}}\right)
+q^{\alpha}D^{\alpha-\beta}(q^{\alpha}x_{2})x_{1}^{-1}\delta\left(\frac{q^{\alpha}x_{2}}{q^{\beta}x_{1}}\right)
                                                               \nonumber\\
&&{} + \frac{1}{q^{-\alpha-\beta}-q^{\alpha+\beta}}
\left[x_{1}^{-1}\delta\left(\frac{q^{\alpha}x_{2}}{q^{-\beta}x_{1}}\right)
-x_{1}^{-1}\delta\left(\frac{q^{-\alpha}x_{2}}{q^{\beta}x_{1}}\right)\right]{\bf c}x_{2}^{-1}
                                                     \nonumber\\
&&{} -\frac{1}{q^{\beta-\alpha}-q^{\alpha-\beta}}
 \left[x_{1}^{-1}\delta\left(\frac{q^{\alpha}x_{2}}{q^{\beta}x_{1}}\right)
  -x_{1}^{-1}\delta\left(\frac{q^{-\alpha}x_{2}}{q^{-\beta}x_{1}}\right)\right]{\bf c}x_{2}^{-1}\nonumber\\
  &=&q^{-\alpha}D^{\alpha+\beta}(q^{-\alpha}x_{2})x_{1}^{-1}\delta\left(\frac{q^{-\alpha-\beta}x_{2}}{x_{1}}\right)
 -q^{\alpha}D^{\alpha+\beta}(q^{\alpha}x_{2})x_{1}^{-1}\delta\left(\frac{q^{\alpha+\beta}x_{2}}{x_{1}}\right)
                                                            \nonumber\\
&&{} -q^{-\alpha}D^{\alpha-\beta}(q^{-\alpha}x_{2})x_{1}^{-1}\delta\left(\frac{q^{\beta-\alpha}x_{2}}{x_{1}}\right)
+q^{\alpha}D^{\alpha-\beta}(q^{\alpha}x_{2})x_{1}^{-1}\delta\left(\frac{q^{\alpha-\beta}x_{2}}{x_{1}}\right)
                                                               \nonumber\\
&&{} + \frac{1}{q^{-\alpha-\beta}-q^{\alpha+\beta}}
\left[x_{1}^{-1}\delta\left(\frac{q^{\alpha+\beta}x_{2}}{x_{1}}\right)
-x_{1}^{-1}\delta\left(\frac{q^{-\alpha-\beta}x_{2}}{x_{1}}\right)\right]{\bf c}x_{2}^{-1}
                                                     \nonumber\\
&&{} -\frac{1}{q^{\beta-\alpha}-q^{\alpha-\beta}}
 \left[x_{1}^{-1}\delta\left(\frac{q^{\alpha-\beta}x_{2}}{x_{1}}\right)
  -x_{1}^{-1}\delta\left(\frac{q^{\beta-\alpha}x_{2}}{x_{1}}\right)\right]{\bf c}x_{2}^{-1},  \label{eq:2.8}
   \end{eqnarray}
   where it is understood that
   \begin{eqnarray*}
 &&  \frac{1}{q^{-\alpha-\beta}-q^{\alpha+\beta}}
\left[x_{1}^{-1}\delta\left(\frac{q^{\alpha+\beta}x_{2}}{x_{1}}\right)
-x_{1}^{-1}\delta\left(\frac{q^{-\alpha-\beta}x_{2}}{x_{1}}\right)\right]
=-x_{2}\frac{\partial}{\partial x_{2}}x_{1}^{-1}\delta\left(\frac{x_{2}}{x_{1}}\right),\\
&&  \frac{1}{q^{\beta-\alpha}-q^{\alpha-\beta}}
\left[x_{1}^{-1}\delta\left(\frac{q^{\alpha-\beta}x_{2}}{x_{1}}\right)
-x_{1}^{-1}\delta\left(\frac{q^{\beta-\alpha}x_{2}}{x_{1}}\right)\right]
=-x_{2}\frac{\partial}{\partial x_{2}}x_{1}^{-1}\delta\left(\frac{x_{2}}{x_{1}}\right)
   \end{eqnarray*}
 for $\alpha+\beta=0$ and for $\alpha-\beta=0$, respectively. (Recall Remarks \ref{rlimit-cases} and \ref{rlimit2}.)
 It can be readily seen that the defining relation (\ref{eDbracket}) is equivalent to (\ref{eDbracket-generating}).

 Motivated by this observation we make the following modification:

   \begin{definition}
 {\em For $\alpha\in \Z$, we set
 \begin{eqnarray}
\tilde{D}^{\alpha}(x)=\begin{cases}D^{\alpha}(x)&\ \ \mbox{ if }\alpha=0\\
D^{\alpha}(x)-\frac{1}{q^{-\alpha}-q^{\alpha}}{\bf c}x^{-1}&\ \ \mbox{ if }\alpha\ne 0.
\end{cases}
 \end{eqnarray}}
   \end{definition}

 Then we immediately have:

  \begin{lem}\label{tildeD-characterization}
 The defining relations of $D$ are equivalent to
 \begin{eqnarray}
 \tilde{D}^{-\alpha}(x)=-\tilde{D}^{\alpha}(x),
 \end{eqnarray}
 \begin{eqnarray}\label{eDbracket-new}
 &&[\tilde{D}^{\alpha}(x_{1}),\tilde{D}^{\beta}(x_{2})] \nonumber\\
 &=&q^{-\alpha}\tilde{D}^{\alpha+\beta}(q^{-\alpha}x_{2})x_{1}^{-1}\delta\left(\frac{q^{-\alpha-\beta}x_{2}}{x_{1}}\right)
 -q^{\alpha}\tilde{D}^{\alpha+\beta}(q^{\alpha}x_{2})x_{1}^{-1}\delta\left(\frac{q^{\alpha+\beta}x_{2}}{x_{1}}\right)
                                                            \nonumber\\
&&{} -q^{-\alpha}\tilde{D}^{\alpha-\beta}(q^{-\alpha}x_{2})x_{1}^{-1}\delta\left(\frac{q^{\beta-\alpha}x_{2}}{x_{1}}\right)
+q^{\alpha}\tilde{D}^{\alpha-\beta}(q^{\alpha}x_{2})x_{1}^{-1}\delta\left(\frac{q^{\alpha-\beta}x_{2}}{x_{1}}\right)
                                                               \nonumber\\
&&+ \left(\delta_{\alpha-\beta,0}-\delta_{\alpha+\beta,0} \right)\frac{\partial}{\partial x_{2}}
x_{1}^{-1}\delta\left(\frac{x_{2}}{x_{1}}\right){\bf c}
   \end{eqnarray}
   for $\alpha,\beta\in \Z$.
 \end{lem}

\begin{definition}\label{restricted}
{\em A $D$-module $W$ is said to be {\em restricted} if
                    for any $\alpha\in \mathbb{Z}$ and $w\in W,$ $D^{\alpha}(n)w = 0 $
                    for $n$ sufficiently large, or equivalently, if
                    $D^{\alpha}(x)\in {\mathcal{E}}(W)$ for any $\alpha\in \Z$, or equivalently, if
                    $\tilde{D}^{\alpha}(x)\in {\mathcal{E}}(W)$ for any $\alpha\in \Z$.
                    We say a $D$-module W is of {\em level} $\ell\in \C$
                if the central element ${\bf c}$ acts as scalar $\ell.$}
\end{definition}

Furthermore, for $ \alpha,r\in\mathbb{Z}$ we set
\begin{eqnarray}\label{eDalpha-r}
D^{\alpha,r}(x) = q^{r}\tilde{D}^{\alpha}(q^{r}x).
\end{eqnarray}
By  (\ref{eDbracket-new}), we have
\begin{eqnarray}\label{eq:2.9}
&&[D^{\alpha,r}(x_{1}),D^{\beta,s}(x_{2})] \nonumber\\
&=& q^{r+s}[\tilde{D}^{\alpha}(q^{r}x_{1}),\tilde{D}^{\beta}(q^{s}x_{2})]         \nonumber\\
        & = &D^{\alpha+\beta,-\alpha+s}(x_{2})x_{1}^{-1}\delta\left(\frac{q^{-\alpha+s}x_{2}}{q^{\beta+r}x_{1}}\right)
              -D^{\alpha+\beta,\alpha+s}(x_{2})x_{1}^{-1}\delta\left(\frac{q^{\alpha+s}x_{2}}{q^{-\beta+r}x_{1}}\right)
                                                              \nonumber\\
         &&{} -D^{\alpha-\beta,-\alpha+s}(x_{2})x_{1}^{-1}\delta\left(\frac{q^{-\alpha+s}x_{2}}{q^{-\beta+r}x_{1}}\right)
         +D^{\alpha-\beta,\alpha+s}(x_{2})x_{1}^{-1}\delta\left(\frac{q^{\alpha+s}x_{2}}{q^{\beta+r}x_{1}}\right)
                                                                   \nonumber\\
       &&{} +\left(\delta_{\alpha-\beta,0}-\delta_{\alpha+\beta,0} \right)\frac{\partial}{\partial x_{2}}
x_{1}^{-1}\delta\left(\frac{q^{s-r}x_{2}}{x_{1}}\right){\bf c}
\end{eqnarray}
for $\alpha,\beta,r,s\in \Z$.

This observation leads us to the following characterization of the Lie algebra $D$:

\begin{prop}\label{psecond-def-D} Lie algebra $D$ is isomorphic to the Lie algebra $L$ with generators ${\bf c}$ and $D^{\alpha,r}(n)$
for $\alpha,r,n\in \Z$, subject to relations $[{\bf c}, L]=0$,
\begin{eqnarray}
D^{-\alpha,r}(x)=-D^{\alpha,r}(x), \ \ \   D^{\alpha,r+s}(x)=q^{s}D^{\alpha,r}(q^{s}x)\label{einvariance}
\end{eqnarray}
for $\alpha,r,s\in \Z$, and subject to the relation (\ref{eq:2.9}),
 where $D^{\alpha,r}(x)=\sum_{n\in\Z}D^{\alpha,r}(n)x^{-n-1}$.
\end{prop}

\begin{proof} Given Lie algebra $D$, with $D^{\alpha,r}(x)$ defined in (\ref{eDalpha-r}),
we see that (\ref{einvariance}) and (\ref{eq:2.9}) hold.
It follows that there is a natural Lie algebra homomorphism $\theta$ from $L$ onto $D$,
sending $D^{\alpha,r}(x)$ to $q^{r}\tilde{D}^{\alpha}(q^{r}x)$ for $\alpha,r\in \Z$.
On the other hand, for Lie algebra $L$ we have
$$D^{\alpha,r}(x)=q^{r}D^{\alpha,0}(q^{r}x)\  \ \mbox{ and }\  D^{-\alpha,0}(x)=-D^{\alpha,0}(x)$$
for $\alpha,r\in \Z$. It is clear that relation (\ref{eq:2.9}) is equivalent to (\ref{eDbracket-new})
with $\tilde{D}^{\alpha}(x)=D^{\alpha,0}(x)$ for $\alpha\in \Z$.
Then there exists a natural Lie algebra homomorphism from $D$ onto $L$,
sending $\tilde{D}^{\alpha}(x)$ to $D^{\alpha,0}(x)$ for $\alpha\in \Z$. Consequently, $\theta$ is a Lie algebra isomorphism.
\end{proof}

Next we associate the $q$-Virasoro algebra $D$ with a specific $\Gamma$-vertex algebra and
its quasi modules. To this end we first introduce a Lie algebra.

\begin{definition}  {\em Let ${\mathfrak{D}}$ be the Lie algebra with generators $d^{\alpha,r}$ for $\alpha, r\in \mathbb{Z}$,
subject to relations
\begin{eqnarray}
&&\hspace{2cm}d^{-\alpha,r}=-d^{\alpha,r},\\
&& [d^{\alpha,r},d^{\beta,s}] =
    \delta_{\alpha+\beta,s-r}d^{\alpha+\beta,-\alpha+s}
   -\delta_{\alpha+\beta,r-s}d^{\alpha+\beta,\alpha+s}
                         \nonumber\\
   &&\hspace{2.2cm}
   -\delta_{\alpha-\beta,s-r}d^{\alpha-\beta,-\alpha+s}
   +\delta_{\alpha-\beta,r-s}d^{\alpha-\beta,\alpha+s}  \label{eq:2.11}
\end{eqnarray}
for $\alpha,\beta, r,s\in \Z$.}
\end{definition}

It is clear from the defining relations that  ${\mathfrak{D}}$ is  linearly spanned by vectors
 $d^{\alpha,r}$ for $\alpha,  r\in\mathbb{Z}$ with $\alpha\ge 1$.
One can show directly that  these vectors actually form a basis of ${\mathfrak{D}}$, following the line of Remark \ref{rDbasis}.
In the following,
we shall show that Lie algebra ${\mathfrak{D}}$ is  isomorphic to a subalgebra of Lie algebra $\mathfrak{gl}_{\infty}$,
which will immediately imply that ${\mathfrak{D}}$ has the desired basis.

Recall that
$\mathfrak{gl}_{\infty}$ is the (Lie) algebra of doubly infinite complex matrices with only finitely many nonzero entries.
For $m,n\in \mathbb{Z}$, let $E_{m,n}$ denote the matrix whose only nonzero entry is the $(m,n)$-entry which is $1$.
Then  $E_{m,n}$ for $m,n\in \mathbb{Z}$ form a basis of $\mathfrak{gl}_{\infty}$, where
$$[E_{m,n}, E_{p,q}]=E_{m,n}\cdot E_{p,q}-E_{p,q}\cdot E_{m,n}=\delta_{n,p}E_{m,q}-\delta_{q,m}E_{p,n}.$$
Equip $\mathfrak{gl}_{\infty}$ with a bilinear form $\langle\cdot,\cdot\rangle$  defined by
\begin{eqnarray}
\langle E_{m,n},E_{r,s}\rangle =\delta_{m,s}\delta_{n,r}\ \ \   \mbox{ for }m,n,r,s\in \Z.
\end{eqnarray}
This form is symmetric, associative (invariant), and non-degenerate.
Set
\begin{eqnarray}
\mathcal{A}=\mbox{span}\{E_{m,n}\   |\   m,n\in \mathbb{Z}\ \mbox{with}\   m+n\in 2\mathbb{Z}\},
\end{eqnarray}
which is a (Lie) subalgebra of $\mathfrak{gl}_{\infty}$. For $\alpha,m\in \Z$, set
$$G_{\alpha,m}=E_{\alpha+m,m-\alpha}\in \mathcal{A}.$$
 Then  $G_{\alpha,m}$  for $\alpha,m\in \Z$ form a basis of  $\mathcal{A}$.
Let $\tau$ be the order-$2$ automorphism of Lie algebra
 $\mathfrak{gl}_{\infty}$ defined by
 \begin{eqnarray}\label{etau-def}
 \tau(E_{m,n})=-E_{n,m}\ \ \  \mbox{ for }m,n\in\mathbb{Z}.
 \end{eqnarray}
 Clearly, $\tau$ preserves $\mathcal{A}$ and we have $\tau (G_{\alpha,m})=-G_{-\alpha,m}$ for $\alpha,m\in \Z$.
 Set
 \begin{eqnarray}
 G^{\tau}_{\alpha,m}=G_{\alpha,m}-G_{-\alpha,m}=E_{\alpha+m,m-\alpha}-E_{m-\alpha,m+\alpha}\in {\mathcal{A}}.
 \end{eqnarray}
 Let ${\mathcal{A}}^{\tau}$ denote the Lie subalgebra
 of $\tau$-fixed points in $\mathcal{A}$.
Then
$$G^{\tau}_{-\alpha,m}=-G^{\tau}_{\alpha,m}\  \  \mbox{ for }\alpha,m\in \Z$$
and
 $\{G^{\tau}_{\alpha,m}\  |\   \alpha\geq1, m\in \mathbb{Z}\}$ is a basis of ${\mathcal{A}}^{\tau}$.
By a straightforward calculation we get
 \begin{eqnarray}
 \quad\qquad\qquad [G^{\tau}_{\alpha,r},G^{\tau}_{\beta,s}] &=
   & \delta_{\alpha+\beta,r-s}G^{\tau}_{\alpha+\beta,s+\alpha}-\delta_{\alpha+\beta,s-r}G^{\tau}_{\alpha+\beta,\beta+r}
                         \nonumber\\
   &&{}+ \delta_{\alpha-\beta,s-r}G^{\tau}_{\alpha-\beta,r-\beta}- \delta_{\alpha-\beta,r-s}G^{\tau}_{\alpha-\beta,\alpha+s}
\end{eqnarray}
for $\alpha,\beta,r,s\in \Z$. Furthermore, we have
\begin{eqnarray}
\langle G^{\tau}_{\alpha,r},G^{\tau}_{\beta,s}\rangle
&=&2\left(\delta_{\alpha+\beta,r-s}\delta_{\alpha+\beta,s-r}-\delta_{\alpha-\beta,r-s}\delta_{\alpha-\beta,s-r}\right)\nonumber\\
&=&2\delta_{r,s}\left(\delta_{\alpha+\beta,0}-\delta_{\alpha-\beta,0}\right).
\end{eqnarray}

Therefore, we have proved:

\begin{lem}\label{identification}
Lie algebra ${\mathfrak{D}}$ is isomorphic to the Lie algebra ${\mathcal{A}}^{\tau}$ with $d^{\alpha,r}$
corresponding to $G^{\tau}_{-\alpha,r}$ $(=-G^{\tau}_{\alpha,r})$ for $\alpha,r \in \Z$  and ${\mathfrak{D}}$
has a basis $\{d^{\alpha,r}\  | \  \alpha,  r\in\mathbb{Z}\   \mbox{ with }\alpha\ge 1\}$. Furthermore,
the bilinear form $\langle\cdot,\cdot\rangle$ on ${\mathfrak{D}}$, defined by
\begin{eqnarray}\label{ebilinear-form}
\langle d^{\alpha,r},d^{\beta,s}\rangle =\delta_{r,s}\left(\delta_{\alpha-\beta,0}-\delta_{\alpha+\beta,0}\right)
\end{eqnarray}
for $ \alpha,\beta,r, s\in\mathbb{Z}$, is
symmetric and invariant.
\end{lem}

Now, we equip Lie algebra ${\mathfrak{D}}$ with the symmetric invariant bilinear form $\langle \cdot,\cdot\rangle$ defined as in Lemma \ref{identification}.
Then we have an affine Lie algebra
   $$\widehat{\mathfrak{D}} = {\mathfrak{D}}\otimes\mathbb{C}[t,t^{-1}]\oplus \C {\bf c},  $$
where for $\alpha,\beta,r,s\in \Z$,
\begin{eqnarray}
 &&[d^{\alpha,r}(x_{1}),d^{\beta,s}(x_{2})]     \nonumber\\
  &=&[d^{\alpha,r},d^{\beta,s}](x_{2})x_{1}^{-1}\delta\left(\frac{x_{2}}{x_{1}}\right)
  +\langle d^{\alpha,r},d^{\beta,s}\rangle \frac{\partial}{\partial x_{2}}x_{1}^{-1}\delta\left(\frac{x_{2}}{x_{1}}\right){\bf c}\nonumber\\
  &=& \delta_{\alpha+\beta,s-r}d^{\alpha+\beta,-\alpha+s}(x_{2})x_{1}^{-1}\delta\left(\frac{x_{2}}{x_{1}}\right)
   -\delta_{\alpha+\beta,r-s}d^{\alpha+\beta,\alpha+s}(x_{2})x_{1}^{-1}\delta\left(\frac{x_{2}}{x_{1}}\right)
                         \nonumber\\
   &&-\delta_{\alpha-\beta,s-r}d^{\alpha-\beta,-\alpha+s}(x_{2})x_{1}^{-1}\delta\left(\frac{x_{2}}{x_{1}}\right)
   +\delta_{\alpha-\beta,r-s}d^{\alpha-\beta,\alpha+s}(x_{2})x_{1}^{-1}\delta\left(\frac{x_{2}}{x_{1}}\right)\nonumber\\
 && +\delta_{r,s}\left(\delta_{\alpha-\beta,0}-\delta_{\alpha+\beta,0}\right)
 \frac{\partial}{\partial x_{2}}x_{1}^{-1}\delta\left(\frac{x_{2}}{x_{1}}\right){\bf c} .  \label{eq:2.13}
\end{eqnarray}
Recall that for $\alpha,r\in \Z$,
$$d^{\alpha,r}(x)=\sum\limits_{n\in\mathbb{Z}}d^{\alpha,r}_{n}x^{-n-1},$$
 where $d^{\alpha,r}_{n} $ denotes $d^{\alpha,r}\otimes t^{n}.$

Let  $\ell$ be a complex number, which is fixed for the rest of this section.
Given the affine Lie algebra $\widehat{\mathfrak{D}}$ and the complex number $\ell$,
we have a $\Z$-graded vertex algebra $V_{\widehat{\mathfrak{D}}}(\ell,0)$.
 Recall that as the underlying space,
   $$ V_{\widehat{\mathfrak{D}}} (\ell,0)=
U(\widehat{\mathfrak{D}})\otimes_{U({\mathfrak{D}}\otimes \C[t]+\C {\bf c})}\mathbb{C}_{\ell},   $$
where $\mathbb{C}_{\ell}$ denotes the one-dimensional $({\mathfrak{D}}\otimes \C[t]+\C {\bf c})$-module $\C$ with ${\mathfrak{D}}\otimes \C[t]$
   acting trivially and with ${\bf c}$ acting as scalar $\ell$, and that $\textbf{1} = 1\otimes 1$,
$Y(d^{\alpha,r},x) = d^{\alpha,r}(x)$ for $\alpha,r\in \Z$.
Furthermore,
$V_{\widehat{\mathfrak{D}}}(\ell,0)_{(1)}={\mathfrak{D}}$
is a generating subspace of $V_{\widehat{\mathfrak{D}}}(\ell,0)$, where
${\mathfrak{D}}$ is identified as a subspace of $V_{\widehat{\mathfrak{D}}}(\ell,0)$ through the linear map
$a\in {\mathfrak{D}}\mapsto a(-1)\textbf{1}.$

Furthermore, we have:

\begin{lem}\label{automorphism1}
 For $m\in\mathbb{Z}$, there exists an automorphism $\sigma_{m}$
               of the $\Z$-graded vertex algebra $ V_{\widehat{\mathfrak{D}}}(\ell,0)$, which is uniquely determined by
               $\sigma_{m}(d^{\alpha,r})=d^{\alpha,m+r}$ for $\alpha,r\in\mathbb{Z}$.
\end{lem}

\begin{proof}  Let $m\in\mathbb{Z}$. Define a linear endomorphism $\sigma'_{m}$  of ${\mathfrak{D}}$
             by $\sigma'_{m}(d^{\alpha,r})=d^{\alpha,m+r}$ for $\alpha,r\in\mathbb{Z}$.
             It can be readily seen that $\sigma'_{m}$ is an automorphism of Lie algebra ${\mathfrak{D}}$, which preserves the bilinear form
             $\langle\cdot,\cdot\rangle$. By Lemma \ref{lautomorphism}, $\sigma'_{m}$ extends uniquely to an automorphism $\sigma_{m}$ of
             the $\Z$-graded vertex algebra $V_{\widehat{\mathfrak{D}}}(\ell,0)$.
\end{proof}

Set
\begin{eqnarray}
\Gamma=\{ \sigma_{m}\  | \  m\in \Z\}\subset {\mathrm{Aut}}(V_{\widehat{\mathfrak{D}}}(\ell,0)),
\end{eqnarray}
where $V_{\widehat{\mathfrak{D}}}(\ell,0)$ is viewed as a $\Z$-graded vertex algebra.
It can be readily seen that the map $\Z\ni m\mapsto \sigma_{m}\in \Gamma$ is a group isomorphism.
We define two group homomorphisms
$$R: \  \Gamma \rightarrow \ {\mathrm{GL}}(V_{\widehat{\mathfrak{D}}}(\ell,0))
\   \   \mbox{ and }\  \  \varphi: \  \Gamma\  \rightarrow \   \mathbb{C}^{\times}$$
 by
$$ \varphi(\sigma_{m}) = q^{m} \  \  \mbox{and}\  \  R_{\sigma_{m}} = \varphi(\sigma_{m})^{-L(0)}\sigma_{m}=q^{-mL(0)}\sigma_{m}
  \   \  \mbox{for}\;\; m\in\Z,$$
 where $L(0)$ is the linear operator on $V_{\widehat{\mathfrak{D}}}(\ell,0)$ defined by $L(0)v=nv$
 for $v\in V_{\widehat{\mathfrak{D}}}(\ell,0)_{(n)}$ with $n\in \mathbb{Z}$.
In view of Remark \ref{rgradedva-gammava},  equipped with group homomorphisms $R$ and $\varphi$,
the $\Z$-graded vertex algebra $V_{\widehat{\mathfrak{D}}}(\ell,0)$ becomes a $\Gamma$-vertex algebra.

For the rest of this section, we consider $V_{\widehat{\mathfrak{D}}}(\ell,0)$ as a $\Gamma$-vertex algebra with $\Gamma=\Z$ as above.
Now, we are in a position to present our main result of this section.

\begin{thm}\label{quasi-main}
 Let $W$ be a restricted $D$-module of level $\ell$. Then there exists a quasi
                 $V_{\widehat{\mathfrak{D}}}(\ell,0)$-module
                 structure $Y_{W}(\cdot,x)$ on $W$, which is uniquely determined by
                 $$Y_{W}(d^{\alpha,r},x) = D^{\alpha,r}(x)=q^{r}\tilde{D}^{\alpha}(q^{r}x)\  \
                 \mbox{ for }\alpha,r\in\mathbb{Z}.$$
                 On the other hand, let $(W,Y_{W})$ be a quasi $V_{\widehat{\mathfrak{D}}}(\ell,0)$-module.
                     Then  $W$ is a restricted $D$-module of level $\ell$
                  with $\tilde{D}^{\alpha}(x)=Y_{W}(d^{\alpha,0},x)$
                  for $\alpha\in\mathbb{Z}.$
 \end{thm}

\begin{proof} We shall apply Theorem 4.9 of \cite{Li2} by showing that Lie algebra $D$ is isomorphic to the $\Gamma$-covariant algebra of
$\widehat{\mathfrak{D}}$ as defined therein. Recall that
$$\Gamma=\{ \sigma_{m} \  | \ m\in \Z\}\subset {\mathrm{Aut}}({\mathfrak{D}},\langle\cdot,\cdot\rangle),$$
where $\sigma_{m}(d^{\alpha,r})=d^{\alpha,r+m}$ for $m,\alpha,r\in\Z$.
The group $\Gamma$ is canonically isomorphic to $\Z$ with $\sigma_{m}$ corresponding to $m$.
For any given $\alpha,\beta,r,s\in \Z$, from (\ref{eq:2.11}) we see that
$$[\sigma_{m}(d^{\alpha,r}),d^{\beta,s}]=[d^{\alpha,r+m},d^{\beta,s}]= 0$$
for $m\ne \pm (\alpha+\beta)-r+s,\  \pm (\alpha-\beta)-r+s$, and from (\ref{ebilinear-form}) we see that
$$\langle \sigma_{m}(d^{\alpha,r}),d^{\beta,s}\rangle=\langle d^{\alpha,r+m},d^{\beta,s}\rangle= 0$$
for $m\ne s-r$. We are given the linear character $\chi: \Gamma\rightarrow \C^{\times}$ given by
$\chi (\sigma_{m})=q^{m}$ for $m\in \Z$. From \cite{Li2} (Proposition 4.4), we have a Lie algebra $\widehat{\mathfrak{D}}/\Gamma$,
which is constructed as follows: Define a new operation $[\cdot,\cdot]_{\Gamma}$ on the vector space
$$\widehat{\mathfrak{D}}={\mathfrak{D}}\otimes \C[t,t^{-1}]\oplus \C {\bf c}$$ by
$$[a\otimes t^{m}+\lambda {\bf c},b\otimes t^{n}+\mu {\bf c}]_{\Gamma}
=\sum_{k\in \Z}q^{mk}\left([\sigma_{k}(a),b]\otimes t^{m+n}+m\delta_{m+n,0}\langle \sigma_{k}(a),b\rangle {\bf c}\right)$$
for $a,b\in {\mathfrak{D}},\ m,n\in \Z,\ \lambda,\mu\in \C$. Let $J_{\Gamma}$  be the linear span of the elements
$$\sigma_{m} (a)\otimes t^{n}-q^{-mn}(a\otimes t^{n}) \   \  \  \mbox{ for }a\in {\mathfrak{D}},\ m,n\in \Z.$$
It was proved that $J_{\Gamma}$ is a two-sided ideal of the non-associative algebra $(\widehat{\mathfrak{D}},[\cdot,\cdot]_{\Gamma})$ and
the quotient algebra $\widehat{\mathfrak{D}}/J_{\Gamma}$ is a Lie algebra. This particular Lie algebra
 is called the $\Gamma$-covariant Lie algebra of the affine Lie algebra $\widehat{\mathfrak{D}}$,
 denoted by $\widehat{\mathfrak{D}}/\Gamma$. For $\alpha,m,n\in \Z$, set
$$\overline{d}^{\alpha,m}(n)=
d^{\alpha,m}\otimes t^{n}+J_{\Gamma}\in \widehat{\mathfrak{D}}/\Gamma.$$
We have $\overline{d^{-\alpha,m}(n)}=-\overline{d^{\alpha,m}(n)}$ and
$$\overline{d^{\alpha,m+r}}(x)=\sum_{n\in \Z}\overline{(\sigma_{m}d^{\alpha,r})(n)}x^{-n-1}
=\sum_{n\in \Z}q^{-nm}\overline{d^{\alpha,r}(n)}x^{-n-1}=q^{m}\overline{d}^{\alpha,r}(q^{m}x).  $$
Furthermore, using (\ref{eq:2.11}) and (\ref{ebilinear-form}) we get
\begin{eqnarray*}
&&[\overline{d}^{\alpha,r}(x_{1}),\overline{d}^{\beta,s}(x_{2})]\nonumber\\
&=&\sum_{m,n\in \Z}\sum_{k\in \Z}q^{km}\left(\overline{[\sigma_{k}(d^{\alpha,r}),d^{\beta,s}](m+n)}
+m\delta_{m+n,0}\langle \sigma_{k}(d^{\alpha,r}),d^{\beta,s}\rangle {\bf c}\right)x_{1}^{-m-1}x_{2}^{-n-1}\nonumber\\
& = &\overline{d}^{\alpha+\beta,-\alpha+s}(x_{2})x_{1}^{-1}\delta\left(\frac{x_{2}}{q^{\alpha+\beta+r-s}x_{1}}\right)
              -\overline{d}^{\alpha+\beta,\alpha+s}(x_{2})x_{1}^{-1}\delta\left(\frac{x_{2}}{q^{-\alpha-\beta+r-s}x_{1}}\right)
                                                              \nonumber\\
         &&{} -\overline{d}^{\alpha-\beta,-\alpha+s}(x_{2})x_{1}^{-1}\delta\left(\frac{x_{2}}{q^{\alpha-\beta+r-s}x_{1}}\right)
         +\overline{d}^{\alpha-\beta,\alpha+s}(x_{2})x_{1}^{-1}\delta\left(\frac{x_{2}}{q^{\beta-\alpha+r-s}x_{1}}\right)
                                                                   \nonumber\\
       &&{} +\left(\delta_{\alpha-\beta,0}-\delta_{\alpha+\beta,0} \right)\frac{\partial}{\partial x_{2}}
x_{1}^{-1}\delta\left(\frac{q^{s-r}x_{2}}{x_{1}}\right){\bf c}
\end{eqnarray*}
for $\alpha,\beta,r,s\in \Z$.
Then it follows from Proposition \ref{psecond-def-D} that $D$ is isomorphic to $\widehat{\mathfrak{D}}/\Gamma$ with
$D^{\alpha,m}(n)$ corresponding to $\overline{d^{\alpha,m}(n)}$ for $\alpha,m,n\in \Z$.

Notice that $\chi$ is one-to-one as $q$ is not a root of unity  by assumption.
Then all the assertions follow immediately from Theorem 4.9 of \cite{Li2}.
\end{proof}

\section{ Associating $q$-Virasoro algebra with vertex algebras in terms of $\mathbb{Z}$-equivariant $\phi$-coordinated quasi modules}
\def\theequation{3.\arabic{equation}}
\setcounter{equation}{0}

In this section, we prove that for any complex number $\ell$,  the category of restricted $D$-modules of level $\ell$ is isomorphic to that of
$\mathbb{Z}$-equivariant $\phi$-coordinated quasi modules for the vertex algebra $V_{\widehat{\mathfrak{D}}}(\ell,0)$
that was introduced in Section 2.

We first recall some basic notions and results on
$G$-equivariant $\phi$-coordinated quasi modules for vertex algebras from \cite{Li4} and \cite{Li5}.
Set
$$\phi=\phi(x,z)=xe^{z}\in\mathbb{C}[[x,z]],$$
 which is fixed throughout this section.

\begin{definition} {\em Let $V$ be a vertex algebra.
               A {\em $\phi$-coordinated quasi $V$-module} is a vector space $W$ equipped with a
               linear map
               $$Y_{W}(\cdot,x): V\longrightarrow \mathrm{Hom}(W , W((x)))\subset\mathrm{End W}[[x,x^{-1}]],$$
               satisfying the conditions that $Y_{W}(\textbf{1},x)=\textbf{1}_{W}$
               and that for $u,v\in V$, there exists a nonzero polynomial $p(x)$
               such that
               $$ p(x_{1}/x_{2})Y_{W}(u,x_{1})Y_{W}(v,x_{2})\in \mathrm{Hom}(W,W((x_{1},x_{2})))$$
               and
               $$p(e^{z})Y_{W}(Y(u,z)v,x_{2})=(p(x_{1}/x_{2})Y_{W}(u,x_{1})Y_{W}(v,x_{2}))|_{x_{1}=x_{2}e^{z}}.$$}
\end{definition}

\begin{definition}
{\em Let $V$ be a
               vertex algebra and let $G$ be an automorphism group
              equipped with a linear character $\chi: G\longrightarrow \mathbb{C}^{\times}.$
              A {\em G-equivariant $\phi$-coordinated quasi $V$-module} is a
              $\phi$-coordinated quasi $V$-module $(W, Y_{W})$, satisfying the conditions that
              $$Y_{W}(gv,x)=Y_{W}(v,\chi(g)x) \;\;\mbox{for}\;g\in G,\ v\in V,$$
              and that for $u,v\in V$, there exists $p(x)\in\mathbb{C}[x]$ with only zeroes
              in $\chi(G)$ such that
              $$p(x_{1}/x_{2})Y_{W}(u,x_{1})Y_{W}(v,x_{2})\in \mathrm{Hom}(W,W((x_{1},x_{2}))).$$}
\end{definition}

As we need, we also recall the conceptual construction of vertex algebras and their $\phi$-coordinated quasi modules from \cite{Li4}.
Let $W$ be a general vector space.
Let $a(x),b(x)\in\mathcal{E}(W)$. Assume that there exists a nonzero polynomial $p(x)$
                such that
\begin{eqnarray}
p(x/z)a(x)b(z)\in \mathrm{Hom}(W,W((x,z))).  \qquad\;\label{eq:3.1}
\end{eqnarray}
Define $a(x)_{n}^{e}b(x)\in \mathcal{E}(W)$ for $n\in\mathbb{Z}$ in terms of
generating function
$$Y_{\mathcal{E}}^{e}(a(x),z)b(x) = \sum_{n\in\mathbb{Z}}(a(x)_{n}^{e}b(x))z^{-n-1}$$
by
$$Y_{\mathcal{E}}^{e}(a(x),z)b(x) = p(e^{z})^{-1}(p(x_{1}/x)a(x_{1})b(x))|_{x_{1}=xe^{z}}, $$
where $p(x)$ is any nonzero polynomial such that (\ref{eq:3.1}) holds and $p(e^{z})^{-1}$
denotes the inverse of $p(e^{z})$ in $\mathbb{C}((z)).$ (Note that $p(e^{z})$ is a nonzero element of $\C[[z]]$.)

A subspace $U$ of $\mathcal{E}(W)$ is said to be {\em $Y_{\mathcal{E}}^{e}$-closed}
if $a(x)_{n}^{e}b(x)\in U$ for $a(x),b(x)\in U$, $n\in\mathbb{Z}$.
We denote by $\langle U\rangle_{e}$ the smallest $Y_{\mathcal{E}}^{e}$-closed
subspace which contains $U$ and $\textbf{1}_{W}.$

The following result was obtained  in \cite{Li4} (Theorem 5.4 and Proposition 5.3):

\begin{thm}\label{thm3.3}
Let $U$ be a quasi local subset of $\mathcal{E}(W)$. Then
  $(\langle U\rangle_{e},Y_{\mathcal{E}}^{e},1_{W})$ carries the structure of a vertex algebra and  $W$ is a
  $\phi$-coordinated quasi $\langle U\rangle_{e}$-module with
 $Y_{W}(a(x),z) = a(z)$
  for $a(x)\in\langle U\rangle_{e}.$
\end{thm}

We next associate the $q$-Virasoro algebra $D$ with a vertex algebra $V_{\widehat{\mathfrak{D}}}(\ell,0)$  and its
$\mathbb{Z}$-equivariant $\phi$-coordinated quasi modules.
First of all, we modify the generating functions by a shift. For $\alpha\in \Z$, set
\begin{eqnarray}
\overline{D}^{\alpha}(x) =xD^{\alpha}(x)= \sum_{n\in\mathbb{Z}}D^{\alpha}(n)x^{-n}.
\end{eqnarray}
Then the defining commutation relations of $D$ are equivalent to
\begin{eqnarray}
 [\overline{D}^{\alpha}(x_{1}),\overline{D}^{\beta}(x_{2})]  &=&
 \sum_{m,n\in\mathbb{Z}}[D^{\alpha}(n),D^{\beta}(m)]x_{1}^{-n}x_{2}^{-m}
                                                            \nonumber\\
&=&\overline{D}^{\alpha+\beta}(q^{-\alpha}x_{2})\delta\left(\frac{q^{-\alpha}x_{2}}{q^{\beta}x_{1}}\right)
 -\overline{D}^{\alpha+\beta}(q^{\alpha}x_{2})\delta\left(\frac{q^{\alpha}x_{2}}{q^{-\beta}x_{1}}\right)
                                                            \nonumber\\
&&{} -\overline{D}^{\alpha-\beta}(q^{-\alpha}x_{2})\delta\left(\frac{q^{-\alpha}x_{2}}{q^{-\beta}x_{1}}\right)
+\overline{D}^{\alpha-\beta}(q^{\alpha}x_{2})\delta\left(\frac{q^{\alpha}x_{2}}{q^{\beta}x_{1}}\right)
                                                               \nonumber\\
&&{} + \frac{1}{q^{-\alpha-\beta}-q^{\alpha+\beta}}
\left[\delta\left(\frac{q^{\alpha}x_{2}}{q^{-\beta}x_{1}}\right)
-\delta\left(\frac{q^{-\alpha}x_{2}}{q^{\beta}x_{1}}\right)\right]{\bf c}
                                                     \nonumber\\
&&{} -\frac{1}{q^{\beta-\alpha}-q^{\alpha-\beta}}
 \left[\delta\left(\frac{q^{\alpha}x_{2}}{q^{\beta}x_{1}}\right)
  -\delta\left(\frac{q^{-\alpha}x_{2}}{q^{-\beta}x_{1}}\right)\right]{\bf c}, \qquad\qquad\qquad\qquad\qquad\quad \label{eq:3.2}
   \end{eqnarray}
      where as before it is understood that
   \begin{eqnarray*}
 &&  \frac{1}{q^{-\alpha-\beta}-q^{\alpha+\beta}}
\left[\delta\left(\frac{q^{\alpha+\beta}x_{2}}{x_{1}}\right)
-\delta\left(\frac{q^{-\alpha-\beta}x_{2}}{x_{1}}\right)\right]
=-x_{2}\frac{\partial}{\partial x_{2}}\delta\left(\frac{x_{2}}{x_{1}}\right),\\
&&  \frac{1}{q^{\beta-\alpha}-q^{\alpha-\beta}}
\left[\delta\left(\frac{q^{\alpha-\beta}x_{2}}{x_{1}}\right)
-\delta\left(\frac{q^{\beta-\alpha}x_{2}}{x_{1}}\right)\right]
=-x_{2}\frac{\partial}{\partial x_{2}}\delta\left(\frac{x_{2}}{x_{1}}\right)
   \end{eqnarray*}
 for $\alpha+\beta=0$ and for $\alpha-\beta=0$, respectively. (Recall Remarks \ref{rlimit-cases} and \ref{rlimit2}.)

\begin{definition}
 {\em For $\alpha\in \Z$, we set
 \begin{eqnarray}
\hat{D}^{\alpha}(x)=x\tilde{D}^{\alpha}(x)=\begin{cases}xD^{\alpha}(x)&\ \ \mbox{ if }\alpha=0\\
xD^{\alpha}(x)-\frac{1}{q^{-\alpha}-q^{\alpha}}{\bf c}&\ \ \mbox{ if }\alpha\ne 0.
\end{cases}
 \end{eqnarray}}
   \end{definition}

Immediately from Lemma \ref{tildeD-characterization}, we have:

\begin{lem}\label{hatD-relations} The defining relations of $D$ are equivalent to that
$$[{\bf c},D]=0, \  \  \  \   \hat{D}^{-\alpha}(x)=-\hat{D}^{\alpha}(x),$$
 and
\begin{eqnarray}\label{emodified}
 [\hat{D}^{\alpha}(x_{1}),\hat{D}^{\beta}(x_{2})]
&=&\hat{D}^{\alpha+\beta}(q^{-\alpha}x_{2})\delta\left(\frac{q^{-\alpha}x_{2}}{q^{\beta}x_{1}}\right)
 -\hat{D}^{\alpha+\beta}(q^{\alpha}x_{2})\delta\left(\frac{q^{\alpha}x_{2}}{q^{-\beta}x_{1}}\right)
                                                            \nonumber\\
&&{} -\hat{D}^{\alpha-\beta}(q^{-\alpha}x_{2})\delta\left(\frac{q^{-\alpha}x_{2}}{q^{-\beta}x_{1}}\right)
+\hat{D}^{\alpha-\beta}(q^{\alpha}x_{2})\delta\left(\frac{q^{\alpha}x_{2}}{q^{\beta}x_{1}}\right)
                                                               \nonumber\\
&&{} +\left(\delta_{\alpha-\beta,0}-\delta_{\alpha+\beta,0}\right)x_{2}\frac{\partial}{\partial x_{2}}
\delta\left(\frac{x_{2}}{x_{1}}\right){\bf c}
   \end{eqnarray}
   for $\alpha,\beta\in \Z$.
\end{lem}

Recall from Section 2 that
$$\Gamma=\{ \sigma_{m} \  |  \  m\in \Z\}\subset \Aut (V_{\widehat{\mathfrak{D}}}(\ell,0)),$$
where
$$\sigma_{m}(d^{\alpha,r})=d^{\alpha,r+m}\  \  \  \mbox{ for }\alpha,r, m\in \Z.$$
We now view $\mathbb{Z}$ as an automorphism group of $ V_{\widehat{\mathfrak{D}}}(\ell,0)$ by identifying $m$ with $\sigma_{m}$ for $m\in\mathbb{Z}$.
Define a linear character $\chi_{q}: \mathbb{Z} \rightarrow \mathbb{C}^{\times}$ by
\begin{eqnarray*}
 \chi_{q}(m)=q^{m} \quad\textmd{for}\; m\in\mathbb{Z}.
\end{eqnarray*}

\begin{definition}  {\em A subset $U$ of $\mathcal{E}(W)$ is said to be {\em $\chi_{q}(\mathbb{Z})$-quasi local}
                 if for any $a(x),b(x)\in U$,
                  there exists a nonzero polynomial
                 $$f(x_{1},x_{2})\in\langle(x_{1}-\alpha x_{2})\  | \  \alpha\in\chi_{q}(\mathbb{Z})\rangle\subset\mathbb{C}[x_{1},x_{2}],$$
                 such that $f(x_{1},x_{2})a(x_{1})b(x_{2})=f(x_{1},x_{2})b(x_{2})a(x_{1}). $}
\end{definition}

As the first main result of this section we have:

\begin{thm}\label{phi-module-part1}
 Let $W$ be a restricted $D$-module of level $\ell$.
                  Then there exists a $\mathbb{Z}$-equivariant
                  $\phi$-coordinated quasi $V_{\widehat{\mathfrak{D}}}(\ell,0)$-module
                 structure $Y_{W}(\cdot,x)$ on $W$, which is uniquely
                 determined by
                 $$Y_{W}(d^{\alpha,r},x) = \hat{D}^{\alpha}(q^{r}x)
                   \  \  \  \mbox{ for }\alpha,r\in\mathbb{Z}.$$
\end{thm}

 \begin{proof} Since ${\mathfrak{D}}$ generates $V_{\widehat{\mathfrak{D}}}(\ell,0)$ as a vertex algebra,
              the uniqueness is clear. We now establish the existence.
              Set
             $ U_{W}=\{1_{W}\}\cup\{\hat{D}^{\alpha}(q^{r}x)
                   \  |  \  \alpha,r\in\mathbb{Z}\} \subset\mathcal{E}(W)$.
                 For $ \alpha,r\in\mathbb{Z}$,    set $\hat{D}^{\alpha,r}(x) = \hat{D}^{\alpha}(q^{r}x) $.
                 We have
           \begin{eqnarray*}
                \hat{D}^{-\alpha,r}(x) =x\tilde{D}^{-\alpha}(q^{r}x)=-x\tilde{D}^{\alpha}(q^{r}x)=-\hat{D}^{\alpha,r}(x).
           \end{eqnarray*}
                  Let $ \alpha,r,\beta,s\in\mathbb{Z}$. Using (\ref{emodified}) we get
             \begin{eqnarray}
          && [\hat{D}^{\alpha,r}(x_{1}),\hat{D}^{\beta,s}(x_{2})] \nonumber\\
          &=& [\hat{D}^{\alpha}(q^{r}x_{1}),\hat{D}^{\beta}(q^{s}x_{2})]         \nonumber\\
        & =& \hat{D}^{\alpha+\beta,-\alpha+s}(x_{2})\delta\left(\frac{q^{-\alpha-\beta-r+s}x_{2}}{x_{1}}\right)
              -\hat{D}^{\alpha+\beta,\alpha+s}(x_{2})\delta\left(\frac{q^{\alpha+\beta-r+s}x_{2}}{x_{1}}\right)
                                                              \nonumber\\
         &&{} -\hat{D}^{\alpha-\beta,-\alpha+s}(x_{2})\delta\left(\frac{q^{\beta-\alpha-r+s}x_{2}}{x_{1}}\right)
         +\hat{D}^{\alpha-\beta,\alpha+s}(x_{2})\delta\left(\frac{q^{\alpha-\beta-r+s}x_{2}}{x_{1}}\right)
                                                                   \nonumber\\
       &&{} +\ell (\delta_{\alpha-\beta,0}-\delta_{\alpha+\beta,0})
          \left(x_{2}\frac{\partial}{\partial x_{2}}\right)\delta\left(\frac{q^{s-r}x_{2}}{x_{1}}\right).\label{eq:3.6}
\end{eqnarray}
From this it follows that $U_{W}$ is a quasi local subset of $\mathcal{E}(W)$.
     By Theorem \ref{thm3.3}, $U_{W}$ generates a vertex algebra $\langle U_{W}\rangle_{e}$
     under the vertex operator operation $Y_{\mathcal{E}}^{e}$ with $W$ as a $\phi$-coordinated quasi module, where
       $Y_{W}(a(x),z) = a(z)$ for $a(x)\in \langle U_{W}\rangle_{e}.$

With (\ref{eq:3.6}), by using the Lemma 4.13 or Proposition 4.14 of \cite{Li5}, we have

 $$\hat{D}^{\alpha,r}(x)_{n}^{e}\hat{D}^{\beta,s}(x) = 0   \;\;\mbox{for}\; n\geq 2,$$
 $$\hat{D}^{\alpha,r}(x)_{1}^{e}\hat{D}^{\beta,s}(x) =\ell (\delta_{\alpha-\beta,0}-\delta_{\alpha+\beta,0})\delta_{s-r,0}1_{W},$$
  \begin{eqnarray*}
 \qquad \hat{D}^{\alpha,r}(x)_{0}^{e}\hat{D}^{\beta,s}(x)&
= &\delta_{\alpha+\beta,s-r} \hat{D}^{\alpha+\beta,-\alpha+s}(x)
   -\delta_{\alpha+\beta,r-s}\hat{D}^{\alpha+\beta,\alpha+s}(x)
                                  \nonumber\\
  &&{} -\delta_{\beta-\alpha,r-s}\hat{D}^{\alpha-\beta,-\alpha+s}(x)
   +\delta_{\alpha-\beta,r-s}\hat{D}^{\alpha-\beta,\alpha+s}(x).
\end{eqnarray*}
(Note that as $q$ is not a root of unity, for $\gamma\in \Z$, $q^{\gamma}=1$ if and only if $\gamma=0$.)
Then by Borcherds' commutator formula we have
\begin{eqnarray}
&&[Y_{\mathcal{E}}^{e}(\hat{D}^{\alpha,r}(x),x_{1}),Y_{\mathcal{E}}^{e}(\hat{D}^{\beta,s}(x),x_{2})] \nonumber\\
&=&\sum_{n\geq 0}Y_{\mathcal{E}}^{e}(\hat{D}^{\alpha,r}(x)_{n}^{e}\hat{D}^{\beta,s}(x),x_{2})\frac{1}{n!}(\frac{\partial}{\partial x_{2}})^{n}
                           x_{1}^{-1}\delta\left(\frac{x_{2}}{x_{1}}\right)  \nonumber\\
&=&\left( \delta_{\alpha+\beta,s-r}Y_{\mathcal{E}}^{e}(\hat{D}^{\alpha+\beta,-\alpha+s}(x),x_{2})
   -\delta_{\alpha+\beta,r-s}Y_{\mathcal{E}}^{e}(\hat{D}^{\alpha+\beta,\alpha+s}(x),x_{2})\right)x_{1}^{-1}\delta\left(\frac{x_{2}}{x_{1}}\right)
                         \nonumber\\
   &&-\left(\delta_{\alpha-\beta,s-r}Y_{\mathcal{E}}^{e}(\hat{D}^{\alpha-\beta,-\alpha+s}(x),x_{2})
   -\delta_{\alpha-\beta,r-s}Y_{\mathcal{E}}^{e}(\hat{D}^{\alpha-\beta,\alpha+s}(x),x_{2})\right)x_{1}^{-1}\delta\left(\frac{x_{2}}{x_{1}}\right)  \nonumber\\
   &&+\delta_{r,s}(\delta_{\alpha-\beta,0}-\delta_{\alpha+\beta,0})
       \frac{\partial}{\partial x_{2}}x_{1}^{-1}\delta\left(\frac{x_{2}}{x_{1}}\right)\ell\cdot 1_{W}.\label{eq:3.7}
\end{eqnarray}
We also have invariance property
$$Y_{\mathcal{E}}^{e}(\hat{D}^{-\alpha,r}(x),x_{1})=Y_{\mathcal{E}}^{e}(-\hat{D}^{\alpha,r}(x),x_{1})
=-Y_{\mathcal{E}}^{e}(\hat{D}^{\alpha,r}(x),x_{1})$$
for $\alpha,r\in \Z$.
 It follows that $\langle U_{W}\rangle_{e}$ is a $\widehat{\mathfrak{D}}$-module of level $\ell$
 with $d^{\alpha,r}(z)$ acting as $Y_{\mathcal{E}}^{e}(\hat{D}^{\alpha,r}(x),z)$
 for $\alpha,r\in\mathbb{Z}.$
 Furthermore, we have
 $\hat{D}^{\alpha,r}(x)_{n}^{e}\textbf{1}_{W}=0$
 for $\alpha,r\in\mathbb{Z},n\in\mathbb{N}.$
 From the construction of $V_{\widehat{\mathfrak{D}}}(\ell,0)$, we see that
 there exists a $\widehat{\mathfrak{D}}$-module homomorphism $\rho$ from $V_{\widehat{\mathfrak{D}}}(\ell,0)$ to $\langle U_{W}\rangle_{e}$,
 sending $\textbf{1}$ to $\textbf{1}_{W}$.
 For $\alpha,r\in\mathbb{Z}$, we have
 $$\rho(d^{\alpha,r})=\rho(d^{\alpha,r}_{-1}\textbf{1})
 =\hat{D}^{\alpha,r}(x)_{-1}\textbf{1}_{W}=\hat{D}^{\alpha,r}(x)\in \langle U_{W}\rangle_{e},$$
 $$\rho(d^{\alpha,r}_{n}v)=\rho(d^{\alpha,r}(n)v)=\hat{D}^{\alpha,r}(x)_{n}\rho(v)=\rho(d^{\alpha,r})_{n}\rho(v)$$
 for $n\in \Z,\ v\in V_{\widehat{\mathfrak{D}}}(\ell,0)$.
Since ${\mathfrak{D}}$ generates $V_{\widehat{\mathfrak{D}}}(\ell,0)$ as a vertex algebra,
it follows that $\rho$ is a homomorphism of vertex algebras.
Consequently, $W$ becomes a
 $\phi$-coordinated quasi $V_{\widehat{\mathfrak{D}}}(\ell,0)$-module
 with $Y_{W}(d^{\alpha,r},x) = \hat{D}^{\alpha,r}(x)$
 for $\alpha,r\in\mathbb{Z}.$

 Furthermore, we have
 $$Y_{W}(\sigma_{m}(d^{\alpha,r}),x)=Y_{W}(d^{\alpha,m+r},x)=\hat{D}^{\alpha,m+r}(x)
 =\hat{D}^{\alpha,r}(q^{m}x)=Y_{W}(d^{\alpha,r},q^{m}x),$$
 and it is clear that $\{Y_{W}(v,x)\  | \  v\in {\mathfrak{D}}\}$ is $\chi_{q}(\mathbb{Z})$-quasi local.
 Then it follows from  \cite{Li5} (Lemma 4.21) that $(W,Y_{W})$ is a $\mathbb{Z}$-equivariant
 $\phi$-coordinated quasi $V_{\widehat{\mathfrak{D}}}(\ell,0)$-module.
 \end{proof}

 On the other hand,  we have:

 \begin{thm}\label{phi-module-D}
 Let $W$ be a $\mathbb{Z}$-equivariant
                  $\phi$-coordinated quasi $V_{\widehat{\mathfrak{D}}}(\ell,0)$-module.
                     Then  $W$ is a restricted $D$-module of level $\ell$
                  with $\hat{D}^{\alpha,r}(x)=Y_{W}(d^{\alpha,r},x)$
                  for $\alpha,r\in\mathbb{Z}.$
\end{thm}

 \begin{proof}               For $\alpha,r,m\in \Z$, we have
              \begin{eqnarray*}
              &&Y_{W}(d^{-\alpha,r},x)=Y_{W}(-d^{\alpha,r},x)=-Y_{W}(d^{\alpha,r},x),\nonumber\\
              &&Y_{W}(d^{\alpha,r+m},x)=Y_{W}(\sigma_{m}(d^{\alpha,r}),x)=Y_{W}(d^{\alpha,r},q^{m}x).
              \end{eqnarray*}
              Let $\alpha,r,\beta,s\in\mathbb{Z}.$
   Note that for $m\in\mathbb{Z}, n \geq 0 ,$
              \begin{eqnarray*}
              (\sigma_{m}d^{\alpha,r})_{n}d^{\beta,s}=d^{\alpha,m+r}_{n}d^{\beta,s}
              &=&\left([d^{\alpha,m+r} , d^{\beta,s}]\otimes t^{n-1}\right)\textbf{1}
                  +\langle d^{\alpha,m+r},d^{\beta,s}\rangle \ell \delta_{n,1}\textbf{1}\nonumber\\
                  &=&\delta_{n,0}[d^{\alpha,m+r} , d^{\beta,s}]+\delta_{n,1}\langle d^{\alpha,m+r},d^{\beta,s}\rangle \ell \textbf{1}.
                  \end{eqnarray*}
              Noticing that $\chi_{q}$ is injective, from \cite{Li5} (Theorem 4.19),  we have
               \begin{align*}
        &[Y_{W}(d^{\alpha,r},x_{1}),Y_{W}(d^{\beta,s},x_{2})] \\
        =&  \mbox{Res}_{x_{0}}\sum_{m\in\mathbb{Z}}Y_{W}\left(Y(\sigma_{m}(d^{\alpha,r}),x_{0})d^{\beta,s},x_{2}\right)
                e^{x_{0}(x_{2}\frac{\partial}{\partial x_{2}})}\delta\left(\frac{\chi_{q}(\sigma_{m})x_{2}}{x_{1}}\right) \\
       =& \sum_{m\in\mathbb{Z}}\sum_{n\ge 0}Y_{W}\left(\sigma_{m}(d^{\alpha,r})_{n}d^{\beta,s},x_{2}\right)
               \frac{1}{n!} \left(x_{2}\frac{\partial}{\partial x_{2}}\right)^{n}\delta\left(\frac{q^{m}x_{2}}{x_{1}}\right) \\
      =& \sum_{m\in\mathbb{Z}}Y_{W}\left([d^{\alpha,r+m},d^{\beta,s}],x_{2}\right)\delta\left(\frac{q^{m}x_{2}}{x_{1}}\right)
                +\ell\langle d^{\alpha,m+r},d^{\beta,s}\rangle \left(x_{2}\frac{\partial}{\partial x_{2}}\right)\delta\left(\frac{q^{m}x_{2}}{x_{1}}\right) \\
       =& Y_{W}(d^{\alpha+\beta,-\alpha+s},x_{2})\delta\left(\frac{q^{-\alpha-\beta-r+s}x_{2}}{x_{1}}\right)
              -Y_{W}(d^{\alpha+\beta,\alpha+s},x_{2})\delta\left(\frac{q^{\alpha+\beta-r+s}x_{2}}{x_{1}}\right)\\
         &- Y_{W}(d^{\alpha-\beta,-\alpha+s},x_{2})\delta\left(\frac{q^{\beta-\alpha-r+s}x_{2}}{x_{1}}\right)
         +Y_{W}(d^{\alpha-\beta,\alpha+s},x_{2})\delta\left(\frac{q^{\alpha-\beta-r+s}x_{2}}{x_{1}}\right)\\
         &+\ell (\delta_{\alpha-\beta,0}-\delta_{\alpha+\beta,0})\left(x_{2}\frac{\partial}{\partial x_{2}}\right)\delta\left(\frac{q^{s-r}x_{2}}{x_{1}}\right).
         \end{align*}
Then it follows from Lemma \ref{hatD-relations} (cf. Proposition \ref{psecond-def-D})
 that $W$ is a $D$-module of level $\ell$ with $D^{\alpha,r}(x)=Y_{W}(d^{\alpha,r},x)$
                  for $\alpha,r\in\mathbb{Z}.$ Since $W$ is a $\phi$-coordinated quasi $V_{\widehat{\mathfrak{D}}}(\ell,0)$-module,
 by definition $Y_{W}(d^{\alpha,k},x)\in\mathcal{E}(W)$ for $\alpha,r\in \Z$. Therefore,
  $W$ is a restricted $D$-module of level $\ell$.
 \end{proof}


\begin{thebibliography}{G-K-K}

\bibitem[BBS]{BBS}
S. Berman, Y. Billig and J. Szmigielski, Vertex operator algebras and the representation theory of toroidal algebras,
{\em Contemporary Math.} {\bf 297}, Amer. Math. Soc., Providence, 2002, 1-26.

\bibitem[FZ]{FZ}
 I. B. Frenkel and Y. Zhu, Vertex operator algebras associated to representations
of affine and Virasoro algebras, {\em Duke Math. J.}  {\bf 66} (1992), 123-168.

\bibitem[G-K-K]{G-K-K}
M. Golenishcheva-kutuzova and V. Kac, $\Gamma$-conformal algebras,
 {\em J. Math. Phys.} {\bf 39} (1998), 2290-2305.

 \bibitem[JL]{JL}
 C. Jiang and H.-S. Li, Associating quantum vertex algebras to Lie algebra $\mathfrak{gl}_{\infty}$,
 {\em J. Algebra} (2014); arXiv:1301.5833v1.

\bibitem[K]{K} V. G. Kac,
\emph{Infinite-dimensional Lie Algebras}, 3rd ed., Cambridge
University Press, Cambridge, 1990.

\bibitem[KL]{KL}
M. Karel and H.-S. Li, Some quantum vertex algebras of Zamolodchikov-Faddeev
type, {\em Commun. Contemp. Math.} {\bf 11} (2009), 829-863.

\bibitem[L1]{Li1}
H.-S. Li, A new construction of vertex algebras and quasi modules
for vertex algebras, {\em Advances in Math.} {\bf 202} (2006), 232-286.

\bibitem[L2]{Li2}
H.-S. Li, On certain generalizations of twisted affine Lie algebras and quasimodules
for $\Gamma$-vertex algebras, {\em J. Pure and Appl. Algebra } {\bf 209} (2007), 853-871.

\bibitem[L3]{Li3} H.-S. Li, Associating quantum vertex algebras to deformed Heisenberg
Lie algebras, {\em Frontiers of Mathematics in China } {\bf 6} (2011), 707-730.

\bibitem[L4]{Li4} H.-S. Li, $\phi$-coordinated quasi-modules for quantum vertex algebras,
{\em Commun. Math. Phys.} {\bf 308} (2011), 703-741.

\bibitem[L5]{Li5} H.-S. Li, $G$-equivariant $\phi$-coordinated quasi-modules for quantum vertex algebras,
 {\em J. Math. Phys.} {\bf 54} (2013), 1-26.

\bibitem[L6]{Li6} H.-S. Li, Nonlocal vertex algebras generated by formal vertex operators,
{\em Selecta Mathematica, New Series} {\bf 11} (2005), 349-397.

\bibitem[L7]{Li7} H.-S. Li, Local systems of vertex operators, vertex subalgebras and modules,
{\em J. Pure Appl. Algebra} {\bf 109} (1996), 143-195.

\bibitem[LL] {LL} J. Lepowsky and H.-S. Li, Introduction to Vertex Operator Algebras
and Their Representations, Progress in Math., Vol. 227,
Birkh\"auser, Boston, 2004.

\bibitem[LTW] {LTW} H.-S. Li, S. Tan and Q. Wang, Toroidal vertex algebras and their modules,
{\em J. Algebra}  {\bf 365} (2012) 50-82.

\bibitem[N]{N} A. Nigro, A q-Virasoro algebra at roots of unity, Free Fermions and
Temperley Lieb Hamiltonians, arXiv:1211.1067v1[math-ph].

\end{thebibliography}
\end{document}